\documentclass[11pt,letterpaper]{amsart}
\usepackage{setspace}
\usepackage{mathrsfs}
\usepackage{amssymb}
\usepackage{latexsym}
\usepackage{amsfonts}
\usepackage{amsmath}
\usepackage{eucal}
\usepackage{bm}
\usepackage{bbm}
\usepackage{graphicx}
\usepackage[english]{varioref}
\usepackage[nice]{nicefrac}
\usepackage[all]{xy}
\usepackage{amsthm}
\usepackage{amssymb,amsthm,upref,amscd}

\def\op{\operatorname}

\def\mmod{\kern-1pt\operatorname{-mod}}

\newtheorem{Thm}{Theorem}[section]
\newtheorem{Lem}[Thm]{Lemma}
\newtheorem{Cor}[Thm]{Corollary}
\newtheorem{Prop}[Thm]{Proposition}

\newtheorem{Rem}[Thm]{Remark}
\newtheorem{Ex}[Thm]{Example}

\newtheorem{Cv}[Thm]{\bf Assumption}

\numberwithin{equation}{section}

\begin{document}

\title[Irreducible Modules with ${\bf B}$-stable Line]{Irreducible Modules of Reductive Groups with Borel-stable Line}

\author{Xiaoyu Chen}
\address{Department of Mathematics, Shanghai Normal University, 100 Guilin Road, Shanghai 200234,
P. R. China}
\email{gauss\_1024@126.com}
%\thanks{}

%\author{}
%\address{}
%\email{}
%\thanks{}

%    General info

\date{Feb 21, 2022}


\begin{abstract}
Let $p$ be a prime number and $\Bbbk=\bar{\mathbb{F}}_p$, the algebraic closure of the finite field $\mathbb{F}_p$ of $p$ elements. Let ${\bf G}$ be a connected reductive group defined over $\mathbb{F}_p$ and ${\bf B}$ be a Borel subgroup of ${\bf G}$ (not necessarily  defined over $\mathbb{F}_p$).  We show that for each (one-dimensional) character $\theta$ of ${\bf B}$ (not necessarily rational), there is a unique (up to isomorphism) irreducible $\Bbbk{\bf G}$-module $\mathbb{L}(\theta)$ containing $\theta$ as a $\Bbbk{\bf B}$-submodule, and moreover, $\mathbb{L}(\theta)$ is isomorphic to a parabolic induction from a finite-dimensional irreducible $\Bbbk{\bf L}$-module for some Levi subgroup ${\bf L}$ of ${\bf G}$. Thus, we have classified and constructed all (abstract) irreducible $\Bbbk{\bf G}$-modules with ${\bf B}$-stable line (i.e. an one-dimensional $\Bbbk{\bf B}$-submodule). As a byproduct, we give a new proof of a result of Borel and Tits on the classification of finite-dimensional irreducible $\Bbbk{\bf G}$-modules.
\end{abstract}

\maketitle

\section*{Introduction}
The classification and construction of certain irreducible modules for a given group is one of the fundamental problems in representation theory. The representation theory of reductive algebraic groups (or  Lie groups) is a main steam of mathematics since it not only has deep connections with algebraic geometry, number theory, etc, but is also interesting in its own right. The representation theory gives important information on the structure of the algebraic (even geometric) objects (groups, algebras, varieties, schemes, $\cdots$).

Earlier attentions has focused on the rational representations of reductive algebraic groups, or representations of finite groups of Lie type. The famous conjecture of Lusztig predicts that character of irreducible rational module of algebraic groups in prime characteristic is determined by Kazhdan--Lusztig polynomials of affine Weyl groups (cf. \cite{L1} and \cite{L2}). Lusztig's conjecture turned out to be true if the characteristic of the base field is large enough (cf. \cite{AJ}). Fiebig showed that Lusztig's conjecture holds if the characteristic of the base field is bigger than an explicit number (cf. \cite{Fie}). Williamson showed that the conjecture is false for small characteristics (cf. \cite{Wil}). The characters of irreducible modules for finite reductive groups in characteristic zero were given in terms of the cohomology of Deligne-Lusztig varieties (cf. \cite{DL}), and the modular version of Deligne--Lusztig theory were given in \cite{BR1} and \cite{BR2}.

But it is still difficult to give a complete classification of all (abstract) irreducible representations for an reductive algebraic group. Even for some ``simple" objects such as $SL_2(\bar{\mathbb{F}}_p)$, the special linear group with coefficients in $\bar{\mathbb{F}}_p$, little was known for its irreducible abstract representations. In \cite{BT}, Borel and Tits gave the classification of finite-dimensional abstract irreducible $\Bbbk'$-representations of an algebraic $\Bbbk$-group, where $\Bbbk'$ is algebraically closed field and $\Bbbk$ is an infinite field.

Recently, some progress has been made in the construction of infinite-dimensional abstract irreducible modules of reductive groups with Frobenius maps (for example, $SL_n(\bar{\mathbb{F}}_p)$, $SO_{2n}(\bar{\mathbb{F}}_p)$, $SO_{2n+1}(\bar{\mathbb{F}}_p)$, $Sp_{2n}(\bar{\mathbb{F}}_p)$). Let ${\bf G}$ be a connected reductive group defined over $\mathbb{F}_p$. In \cite{Xi}, Nanhua Xi constructed certain irreducible modules by taking the union of a system irreducible representations of finite subgroups of ${\bf G}$. In particular, some classical results of Steinberg and Deligne-Lusztig on
complex representations of finite groups of Lie type are extended to infinite reductive algebraic groups with Frobenius maps. It was proved in \cite{Xi} that the direct limit of a system irreducible representations of certain finite subgroups of ${\bf G}$ is still irreducible. In particular, one can take the union of Steinberg modules for finite reductive groups to get an infinite-dimensional Steinberg module.

However, in general, an irreducible $\Bbbk{\bf G}$-module may not be a union of irreducible modules of finite subgroups of ${\bf G}$. It was proved in \cite{Yang} that the infinite-dimensional Steinberg module is always irreducible, although the Steinberg modules for finite reductive groups may not be irreducible in non-defining characteristic. The author has made first attempts to find some irreducible $\Bbbk{\bf G}$-modules for $\Bbbk=\bar{\mathbb{F}}_p$ in unpublished notes \cite{C1} and \cite{C2}. Let ${\bf B}$ be a Borel subgroup of ${\bf G}$, and $\operatorname{tr}$ the trivial character of ${\bf B}$. It was proved in \cite{CD} (resp. \cite{CD2}) that the induced module $\Bbbk{\bf G}\otimes_{\Bbbk{\bf B}}\operatorname{tr}$ has a composition series whose subquotients are indexed by the subsets of simple reflections in Weyl groups if the characteristic of $\Bbbk$ is not equal to $p$ (resp. equal to $p$), although the composition factors of its finite version $\Bbbk G(\mathbb{F}_{p^a})\otimes_{\Bbbk B(\mathbb{F}_{p^a})}\operatorname{tr}$ might be complicated. This is a surprising and new phenomenon in the representation theory of infinite reductive groups. Later, in \cite{CD3}, the composition factors of abstract induced modules from any character of ${\bf B}$ in cross characteristic and some such induced modules in defining characteristic are determined.

Let $\Bbbk=\bar{\mathbb{F}}_p$, the algebraic closure of finite field $\mathbb{F}_p$ of $p$ elements, and $\theta$ be a (one-dimensional) $\Bbbk$-character (possibly not be rational) of ${\bf B}$. In this paper, we show that the induced module $\Bbbk{\bf G}\otimes_{\Bbbk{\bf B}}\theta$ has a unique simple quotient (although this module may have infinitely many composition factors as pointed out in \cite{CD3},  and its ``finite version" $\Bbbk G(\mathbb{F}_{p^a})\otimes_{\Bbbk B(\mathbb{F}_{p^a})}\theta$ may be decomposable in general, with indecomposable summands was given in \cite{SW} and \cite{YY}). Moreover, this simple quotient is isomorphic to a parabolic induction from a finite-dimensional irreducible module for some Levi subgroup of ${\bf G}$. Meanwhile, a result of Borel and Tits (\cite[Theorem 10.3]{BT}) says that any finite-dimensional representations of ${\bf G}$ is isomorphic to a twist tensor product of irreducible rational representations. Thus, our result means that {\bf the rational irreducible modules ``control" the size of all irreducible $\Bbbk{\bf G}$-modules with ${\bf B}$-stable line}.

This paper is organized as follows: In Section 1, we recall some basic facts on the structure theory and representation theory of reductive algebraic groups. In Section 2, we give a series of key observations which lead to the proof of main theorem. Section 3 is devoted to proving a series of main results (Theorem \ref{hd=sim}, Corollary \ref{unique}, Proposition \ref{ifpart}, Theorem \ref{paraind}, Corollary \ref{IndPG}). Combining these results gives a complete classification and construction of $\Bbbk{\bf G}$-modules with ${\bf B}$-stable line.

\smallskip
\noindent{\bf Acknowledgements.} I would like to thank
Professor Nanhua Xi for his helpful suggestions and comments in
writing this paper. I thank
Professor Jianpan Wang and Professor Naihong Hu for their valuable advices. I also thank Junbin Dong for enlightening discussion with him.

%% The correct journal style for \specialsection is all uppercase; a known bug
%% in amsart.cls prevents this, so input must be uppercase until it is fixed.
%\specialsection*{This is a Special Section Head}
%
%%%%%%%%%%%%%%%%%%%%%%%%%%%%%%%%%%%%%%%%%%%%%%%%%%%%%%%%%%%%%%%%%%%%%%%%
%
%%%%%%%%%%%%%%%%%%%%%%%%%%%%%%%%%%%%%%%%%%%%%%%%%%%%%%%%%%%%%%%%%%%%%%%%

\section{Reductive Groups and Representations}
In this section, we briefly recall some notation and basic structure theory of reductive algebraic groups, together with their rational representations and abstract representations (see \cite{Car}, \cite{Jan2}, \cite{Xi} for details).

\subsection{Reductive groups and abstract representations}
Let $p$ be a prime number and $\Bbbk=\bar{\mathbb{F}}_p$, the algebraic closure of the finite field $\mathbb{F}_p$ with $p$ elements. Let ${\bf G}$ be a connected reductive algebraic group defined over $\mathbb{F}_p$ with the standard Frobenius map $F$. For any closed subgroup ${\bf H}$ of ${\bf G}$ defined over $\mathbb{F}_p$, we denote $H_n$ the $\mathbb{F}_{p^{n!}}$-point of ${\bf H}$ (here we use factorial for technical convenience, i.e. $H_m\subset H_n$ if $m<n$). Since $\Bbbk=\bigcup_{n\in\mathbb{N}^*}\mathbb{F}_{p^{n!}}$, we have $H_m\subset H_n$ if $m<n$, and ${\bf H}=\bigcup_{n>0}H_n$. For any finite subset $E$ of ${\bf G}$, we denote $\underline{E}=\sum_{h\in E}h\in\Bbbk{\bf G}$.

Let ${\bf B}$ be a Borel subgroup of ${\bf G}$, and $\theta$ be a (one-dimensional) character of ${\bf B}$ and ${\bf 1}_{{\bf B},\theta}$ be a nonzero vector in the corresponding one-dimensional space. As in \cite{Xi}, define the abstract induced module $\mathbb{M}_{\bf B}(\theta):=\Bbbk{\bf G}\otimes_{\Bbbk{\bf B}}\theta$. For any $g\in{\bf G}$, write $^g{\bf B}=g{\bf B}g^{-1}$. It is easy to check that there is a $\Bbbk{\bf G}$-module isomorphism $\mathbb{M}_{\bf B}(\theta)\simeq\mathbb{M}_{^g{\bf B}}(\theta')$ sending ${\bf 1}_{{\bf B},\theta}$ to $g^{-1}\otimes{\bf 1}_{^g{\bf B},\theta'}$, where $\theta'$ is the character of $^g{\bf B}$ obtained by twisting $\theta$ by the conjugation $\operatorname{Int}g^{-1}:^g{\bf B}\rightarrow{\bf B}$. Therefore, it is enough to study $\mathbb{M}_{\bf B}(\theta)$ for a fixed ${\bf B}$ since any two Borel subgroups are conjugate.

We can assume that ${\bf B}$ is a Borel subgroup defined over $\mathbb{F}_p$ without loss of generality by the above discussion, and ${\bf T}$ a maximal torus in ${\bf B}$ defined over $\mathbb{F}_p$ (Lang--Steinberg Theorem (cf, \cite[10]{St2}) implies the existence of such ${\bf B}$ and ${\bf T}$). In particular, ${\bf U}=R_u({\bf B})$, the unipotent radical of ${\bf B}$, is defined over $\mathbb{F}_p$, and ${\bf B}={\bf T}\ltimes{\bf U}$. We denote $\Phi=\Phi({\bf G};{\bf T})$ the corresponding root system, and $\Phi^+$ (resp. $\Phi^-$) is the set of positive (resp. negative) roots determined by ${\bf B}$. For any $w\in W$, let $\Phi_w^-=\{\alpha\in\Phi^+\mid w\alpha\in\Phi^-\}$ and $\Phi_w^+=\{\alpha\in\Phi^+\mid w\alpha\in\Phi^+\}$. Let $W=N_{\bf G}({\bf T})/{\bf T}$ be the corresponding Weyl group. For each $w\in W$, let $\dot{w}$ be a representative in $N_{\bf G}({\bf T})$. It is well known that $B={\bf B}$ and $N=N_{\bf G}({\bf T})$ form a $BN$-pair of ${\bf G}$. In particular, for each $w\in W$, ${\bf U}$ has two subgroups ${\bf U}_w$ and ${\bf U}_w'$ such that ${\bf U}={\bf U}_w'{\bf U}_w$ and $\dot{w}{\bf U}_w'\dot{w}^{-1}\subset{\bf U}$. The Bruhat decomposition says that ${\bf G}$ is a disjoint union of the double cosets ${\bf B}\dot{w}{\bf B}={\bf U}_{w^{-1}}\times\{\dot{w}\}\times{\bf B}$ $(w\in W)$. The above results also hold for their $\mathbb{F}_{p^{a!}}$-points: $G_a, B_a, U_a, U_{w,a}, U_{w,a}'$ under the following assumption:
\begin{Cv}
We assume that all representatives of the elements of $W$ involved are in $G_1$ without loss of generality. $($Otherwise, one replaces $G_1$ with some $G_n$ containing them. This does no harm to the result.$)$
\end{Cv}
Since $U_{w,a}$ is a finite $p$-group, \cite[Proposition 26]{Se} implies that
\begin{Lem}\label{fixedpt}
Any nonzero $\Bbbk U_{w,a}$-module has a nonzero $U_{w,a}$-fixed point.
\end{Lem}

For each character $\theta$ of ${\bf B}$,  we abbreviate $\mathbb{M}(\theta)$ for $\mathbb{M}_{\bf B}(\theta)$, ${\bf 1}_\theta$ for ${\bf 1}_{{\bf B},\theta}$ (due to the fixed ${\bf B}$), and $x{\bf 1}_\theta$ for $x\otimes{\bf 1}_\theta\in\mathbb{M}(\theta)$. It is clear that ${\bf U}$ acts trivially on ${\bf 1}_\theta$ since ${\bf U}=[{\bf B},{\bf B}]$. The Bruhat decomposition implies that $\mathbb{M}(\theta)=\sum_{w\in W}\Bbbk{\bf U}_{w^{-1}}\dot{w}{\bf 1}_\theta$.

Let $\Delta=\{\alpha_i|i\in I\}$ be the set of simple roots in $\Phi^+$ and $s_i\in W$ $(i\in I)$ corresponding simple reflections. For each $\alpha\in\Phi$, let ${\bf U}_\alpha$ be the corresponding root subgroup of ${\bf G}$. Denote ${\bf G}_i$ the subgroup of ${\bf G}$ generated by ${\bf U}_{\alpha_i}$ and ${\bf U}_{-\alpha_i}$, ${\bf T}_i={\bf T}\cap{\bf G}_i$, and ${\bf T}'$ be the subgroup of ${\bf T}$ generated by ${\bf T}_i$ $(i\in I)$. For each $J\subset I$, let ${\bf G}_J$ be the subgroup of ${\bf G}$ generated by ${\bf G}_i$ ($i\in J$), and $W_J$ be the subgroup of $W$ generated by $s_i$ $(i\in J)$, $W^J$ the set of distinguished representatives of left cosets (i.e. representatives with minimal length in left cosets)  of $W_J$ in $W$, and ${\bf P}_J$ the subgroup of ${\bf G}$ generated by ${\bf B}$ and $\dot{s_i}$ $(i\in J)$. The parabolic version of Bruhat decomposition says that ${\bf G}$ is a disjoint union of the double cosets ${\bf U}_{w^{-1}}\times\{\dot{w}\}\times{\bf P}_J$ $(w\in W^J)$. The same holds for their $\mathbb{F}_{p^{a!}}$-points: $G_a, P_{J,a}$.

For each $\alpha\in\Phi$, we fix an isomorphism $\varepsilon_\alpha: \Bbbk\rightarrow{\bf U}_\alpha$ such that $t\varepsilon_\alpha(c)t^{-1}=\varepsilon_\alpha(\alpha(t)c)$ for any $t\in{\bf T}$ and $c\in\Bbbk$. Set $U_{\alpha,a}=\varepsilon_\alpha(\mathbb{F}_{p^{a!}})$. For each $i\in I$, let ${\bf G}_i$ be the subgroup of ${\bf G}$ generated by ${\bf U}_{\alpha_i}$ and ${\bf U}_{-\alpha_i}$. We fix a homomorphism  $\varphi_i: SL_2(\Bbbk)\rightarrow{\bf G}_i$ such that
$$\varphi_i\left(\begin{array}{cc}1 &\ t\\0 &\ 1\end{array}\right)=\varepsilon_{\alpha_i}(t),~~\varphi_i\left(\begin{array}{cc}1 &\ 0\\t &\ 1\end{array}\right)=\varepsilon_{-\alpha_i}(t),$$
and for $t\in\Bbbk^*$ and $i\in I$, write
$$h_i(t):=\varphi_i\left(\begin{array}{cc}t &\ 0\\0 &\ t^{-1}\end{array}\right),~~\dot{s_i}=\varphi_i\left(\begin{array}{cc}0 &\ 1\\-1 &\ 0\end{array}\right).$$
An easy calculation shows that
\begin{equation}\label{sus}
\dot{s_i}^{-1}\varepsilon_{\alpha_i}(t)\dot{s_i}=
\varepsilon_{\alpha_i}(-t^{-1})\dot{s_i}h_i(t)\varepsilon_{\alpha_i}(-t^{-1})\quad t\in\Bbbk^*.
\end{equation}

We also have the following well known commutator relations. Namely, given $\alpha,\beta\in\Phi^+$,  there exist a total ordering on $\Phi^+$, and integers $c^{mn}_{\alpha \beta}$ such that
\begin{equation}\label{commrel}
[\varepsilon_\alpha(a),\varepsilon_\beta(b)]:=\varepsilon_\alpha(a)\varepsilon_\beta(b)\varepsilon_\alpha(a)^{-1}\varepsilon_\beta(b)^{-1}=
\underset{m,n>0}{\prod} \varepsilon_{m\alpha+n\beta}(c^{mn}_{\alpha \beta}a^mb^n)
\end{equation}
for all $a,b\in\Bbbk$, where the product is over all
$m,n\in\mathbb{N}^*$ such that $m\alpha+n\beta \in \Phi^{+}$, taken
according to the chosen ordering.

For any abelian group $A$, denote $\widehat{A}$ the group of $\Bbbk$-characters of $A$. For convenience, {\it we always regard $\widehat{A}$ as an additive group throughout the paper}. One has the restriction map $\pi:\widehat{\bf T}\rightarrow\widehat{{\bf T}'}$. Since $\Bbbk^*$ is a divisible abelian group, it follows that $\Bbbk^*$ is an injective abelian group. Thus, one obtains

\begin{Lem}\label{surj}
$\pi$ is surjective.
\end{Lem}

\subsection{Rational Representations of ${\bf G}$ and representations of $G_n$}
Let $X({\bf T})$ be the group of rational characters of ${\bf T}$, and $Y({\bf T})$ be the set of algebraic group homomorphisms $\Bbbk^*\rightarrow{\bf T}$. There is a pairing $\langle\cdot,\cdot\rangle: X({\bf T})\times Y({\bf T})\rightarrow\mathbb{Z}$ such that for any $\lambda\in X$ and $\mu\in Y({\bf T})$, $\lambda\circ\mu(t)=t^{\langle\lambda,\mu\rangle}$ $(t\in\Bbbk^*)$. For each $\alpha\in\Phi$, let $\alpha^\vee$ be the coroot in $Y({\bf T})$ such that $\langle\alpha,\alpha^\vee\rangle=2$. A $\lambda\in X({\bf T})$ is called a dominant weight if $\langle\lambda,\alpha_i^\vee\rangle\ge0$ for any $i\in I$. Denote by $X({\bf T})^+$ for the set of dominant weights. For any $\lambda\in X({\bf T})^+$, there is a unique highest weight module $V(\lambda)$ called a Weyl module, which is universal in the sense that if $M$ is a
highest weight module with highest weight $\lambda$, there is a unique, up to scalar, surjection
$V(\lambda)\rightarrow M$. One denotes $\nabla(\lambda)=V(-w_0\lambda)^*$ (this is called a co-standard module). It is well known that $\operatorname{Soc}\nabla(\lambda)$ is simple (denoted $L(\lambda)$), and each irreducible rational ${\bf G}$-module is isomorphic to some $L(\lambda)$ $(\lambda\in X({\bf T})^+)$.

 Let $X_n({\bf T})=\{\lambda\in X({\bf T})^+\mid\langle\lambda,\alpha_i^\vee\rangle<p^{n!}\}$. The classification of irreducible $\Bbbk G_n$-modules is summarized in the following theorem which was first proved in full generality by Steinberg in \cite{St3}.
 \begin{Thm}\label{restrictionthm}
 Assume that ${\bf G}$ is semisimple. Then

 $(\op{i})$ For each $\lambda\in X_n({\bf T})$, $L(\lambda)$ remains irreducible when restricted to $G_n$;

 $(\op{ii})$ Any irreducible $\Bbbk G_n$-module is of the form $L(\lambda)$ for some $\lambda\in X_n({\bf T})$.
 \end{Thm}

\begin{Ex}
\normalfont
In the case ${\bf G}=SL_2(\Bbbk)$, $X({\bf T})$ and $X({\bf T})^+$ are identified with $\mathbb{Z}$ and $\mathbb{N}$, respectively, and for any $n\in\mathbb{N}$, $\nabla(n)$ is identified with space of polynomials in two variables of degree $n$. For each $n\in\mathbb{N}$, By \cite[II Prop 5.2]{Jan1}, there is a basis $v_0,\cdots,v_n$ of $\nabla(n)$ such that
\begin{equation}\label{H0}
\left(\begin{array}{cc}1 &\ t\\0 &\ 1\end{array}\right)v_i=\sum_{j=0}^i\left(i\atop j\right)a^{i-j}v_j\quad\forall t\in\Bbbk,
\end{equation}
and $L(n)$ is spanned by all $v_i$ with $\displaystyle\left(n\atop i\right)\ne0$ $(\operatorname{mod}p)$.

The situation of $PGL_2(\Bbbk)$ is the same as that of $SL_2(\Bbbk)$ except that $X({\bf T})$ and $X({\bf T})^+$ are identified with $2\mathbb{Z}$ and $2\mathbb{N}$, respectively.
\end{Ex}

\section{Key Observations}
In this section, we give some elementary observation to be used in the next section.

For each $n\in\mathbb{N}$, let $f(n)$ be the sum of coefficients in the $p$-adic expansion of $n$. Let $M_n$ be the number of nonzero coefficients in the $p$-adic expansion of $n$.
\begin{Lem}\label{cor33}
Let $r\in\mathbb{N}^*$, $r>1$, and $m,m'\in\mathbb{N}$. Assume that $q=p^r$, $0\le m\le q-1$ and $m'\equiv m$ $(\operatorname{mod}q-1)$. Let $m=\sum_{i=0}^{r-1}m_ip^i$ and $m'=\sum_jm_j'p^j$ be their $p$-adic expansion. Set $r_i'=\sum_{j\equiv i(\operatorname{mod}r)}m'_j$. Then $f(m')\ge f(m)$ and the equality holds if and only if $r_i'=m_i$ for all $i$. In particular, if $f(m')=f(m)$, then $M_{m'}\ge M_m$.
\end{Lem}
\begin{proof}
For each $n\in\mathbb{N}$, write $n=n_{r-1}p^{r-1}+t$ with $0\le t<p^{r-1}$ and define $g(n)=n_{r-1}+f(t)$. We will prove through the following two claims.

\medskip
\noindent{\it Claim 1}: Let $b_i\in\mathbb{N}~(0\le i\le r-1)$ and $n=\sum_{i=0}^{r-1}b_ip^i$.
Then $\sum_{i=0}^{r-1}b_i\ge g(n)$ and the equality holds if and only if $0\le b_i<p$ $(0\le i\le r-2)$.

\smallskip
\noindent{\it Proof of Claim 1}: Write $n=\sum_{i=0}^{r-1}a_ip^i$ with $0\le a_i<p$ $(0\le i\le r-2)$. For each $0<j\le r-1$, let $r_j=\sum_{i=0}^{j-1}a_ip^i$, $r_j'=\sum_{i=0}^{j-1}b_ip^i$, and $q_j=\sum_{i=j}^{r-1}a_ip^{i-j}$, $q_j'=\sum_{i=j}^{r-1}b_ip^{i-j}$. Then $r_j+q_jp^j=r_j'+q_j'p^j$ $(0<j\le r-1)$. Since $0\le a_i<p$ $(0\le i\le j-1)$, we have $r_j\le r_j'$ $(0<j\le r-1)$, and hence $q_j\ge q_j'$ $(0<j\le r-1)$. A simple calculation shows that $b_0-a_0=(q_1-q_1')p$, $b_i-a_i=(q_{i+1}-q_{i+1}')p-(q_i-q_i')$ $(1\le i\le r-2)$, $b_{r-1}-a_{r-1}=-(q_{r-1}-q_{r-1}')$. Taking the sum of the above equations yields $$\sum_{i=0}^{r-1}b_i-g(n)=\sum_{i=0}^{r-1}(b_i-a_i)=(p-1)\sum_{i=1}^{r-1}(q_i-q_i')\ge0,$$
and the equality holds if and only if $q_i=q_i'$ $(1\le i\le r-1)$ which is equivalent to $a_i=b_i$ $(0\le i\le r-1)$, and hence $0\le b_i<p$ $(0\le i\le r-2)$.

\medskip
\noindent{\it Claim 2}: Assume that $0\le m\le q-1$ and $m'\equiv m$ $(\operatorname{mod}q-1)$. Then $g(m')\ge f(m)$ and the equality holds if and only if $m'=m$.

\noindent{\it Proof of Claim 2}: It is clear that $g(n+q)=g(n)+p$ for any $n$ by definition. For each $n\in\mathbb{N}^*$, let $i_n\in\mathbb{N}$ determined by $p^{i_n}|n$ and $p^{i_n+1}\nmid n$. By definition, we have $$g(n-1)=g(n)+(p-1)i_n-1\ge g(n)-1.$$ Thus, we have $$g(n+q-1)\ge g(n+q)-1=g(n)+p-1>g(n).$$
Let $m'=m+k(q-1)$, $k\in\mathbb{N}$. By the above inequality, it follows immediately that $$g(m')=g(m+k(q-1))\ge g(m)=f(m),$$
and the equality holds if and only if $k=0$, and hence $m'=m$.

\medskip
\noindent{\it Proof of lemma}: Let $m'=\sum_iR_iq^i$ be the $q$-adic expansion of $m'$. Then $\sum_iR_i\equiv m$ $(\operatorname{mod}q-1)$. On the other hand, $\sum_iR_i=\sum_{i=0}^{r-1}r_i'p^i$. It follows above two claims that $$f(m')=\sum_{i=0}^{r-1}r_i'\ge g\left(\sum_iR_i\right)\ge f(m),$$
where the first inequality follows from Claim 1 and the second one follows from Claim 2.
Thus, the equality holds if and only if $\sum_{i=0}^{r-1}r_i'p^i=g\left(\sum_iR_i\right)=f(m)$.
The former holds if and only if $0\le r_i'<p$ $(0\le i\le r-2)$ and the latter holds if and only if $\sum_iR_i=m$, i.e. $\sum_{i=0}^{r-1}r_i'p^i=\sum_{i=0}^{r-1}m_i'p^i$. Since $0\le r_i'<p$ $(0\le i\le r-2)$ and $0\le m_i<p$ $(0\le i\le r-1)$, we have $r_i'=m_i$ $(0\le i\le r-1)$ by comparing the coefficients, and hence $M_{m'}\ge M_m$.
\end{proof}

 For each $\sigma\in\widehat{\Bbbk^*}$, by considering its restriction to all $\mathbb{F}_{p^{n!}}$ with $n\in\mathbb{N}^*$, one identifies $\sigma$ with an array $(m_n)_{n\in\mathbb{N}^*}$ such that $0\leq m_n<p^{n!}-1$, and
\begin{equation}\label{mod}
m_i\equiv m_j~(\operatorname{mod}p^{i!}-1)~\mbox{if}~i<j.
\end{equation}
All $m_n$ are characterized by $\sigma(t)=t^{m_n}$ for any $t\in\mathbb{F}_{p^{n!}}$. In other words, we have $$\widehat{\Bbbk^*}\simeq\underleftarrow{\lim}_n~\mathbb{Z}/(p^{n!}-1)\mathbb{Z}$$
as abelian groups.

The following elementary lemma is well known and will be used later.

\begin{Lem}\label{Binomial}
$(\op{i})$ $(${\cite[5.1]{Haboush}}$)$ Assume that $p$ is a prime. Let $m, n$ be two positive integers with $p$-adic expansion
$$m=\sum_{i}a_i p^i~~~~~ \text{and} ~~~~~n=\sum_{i}b_i p^i.$$
Then we have
$$\left(m\atop n\right)\equiv \prod_{i}\left(a_i\atop b_i\right) \quad(\op{mod}p) .$$
In particular, $p|\left(m\atop n\right)$ if and only if there exists $i$ such that
$a_i <b_i$.

\noindent $(\op{ii})$ $(${\cite[Lemma 2.1]{Se2}}$)$ Let $q$ be a power of $p$. Then
$$\sum\limits_{t\in\mathbb{F}_{q}}t^k=\left\{
\begin{array}{ll}
-1 &\ \mbox{if}~(q^a-1)|k~\mbox{and}~k\neq0\\
0 &\ \mbox{otherwise}
\end{array}
\right.,~~~\sum\limits_{t\in\mathbb{F}_{q^a}^*}t^k=\left\{
\begin{array}{ll}
-1 &\ \mbox{if}~(q^a-1)|k\\
0 &\ \mbox{otherwise}
\end{array}
\right..$$
\end{Lem}

Lemma \ref{padic} and \ref{red} below is the key to prove results in the next section.
\begin{Lem}\label{padic}
The following conditions on $\theta=(m_n)_{n\in\mathbb{N}^*}\in\widehat{\Bbbk^*}$ are equivalent:

\noindent$(\operatorname{i})$ For any $r\in\mathbb{N}$, we have $\displaystyle\left({m_s}\atop k(p^{r!}-1)\right)\neq0$ $(\operatorname{mod}p)$ for some $s>r$ and $k\in\mathbb{N}^*$;

\noindent$(\operatorname{ii})$ $f(m_n)$ $(n\in\mathbb{N})$ are unbounded.
\end{Lem}
\begin{proof}
(i) $\Rightarrow$ (ii): Suppose $f(m_n)$ $(n\in\mathbb{N}^*)$ are bounded. Then we can find $r\in\mathbb{N}^*$ such that $f(p^{r!}-1)>f(m_n)$ for any $n\in\mathbb{N}^*$. It follows that $f(m_n)<f(k(p^{r!}-1))$ for any $n,k\in\mathbb{N}^*$ by Lemma \ref{cor33}. This implies that some digit in $p$-adic expansion of $m_n$ is less than the corresponding digit in that of $k(p^{r!}-1)$ and hence $\displaystyle\left({m_n}\atop k(p^{r!}-1)\right)=0$ $(\operatorname{mod}p)$ for any $n,k\in\mathbb{N}^*$ by Lemma \ref{Binomial} (i). This contradicts (i).

\medskip
\noindent(ii) $\Rightarrow$ (i): For any $r\in\mathbb{N}$, (ii) implies that $f(m_s)\ge r!(p^{r!}-1)$ for some $s>r$. Let $m_s=\sum_jm_{sj}p^j$ be its $p$-adic expansion. For any $0\le i<r!$, let $f_i=\sum_{j\equiv i (\operatorname{mod}r!)}m_{sj}$. It is clear that $f(m_s)=\sum_{0\le i<r!}f_i$. It follows that
\begin{equation}\label{31}
 f_{i_0}\ge p^{r!}-1~\mbox{for~some}~0\le i_0<r!.
 \end{equation}
Let $m_s=\sum_tR_tq^t$ be the $q=p^{r!}$-adic expansion of $m_s$. Then
 $$R_t=\sum_{tr!\le j\le(t+1)r!-1}m_{sj}p^{j-tr!}.$$
Since $\operatorname{gcd}(p^{i_0},p^{r!}-1)=1$, there exists an $0\le N<p^{r!}-1$ such that
\begin{equation}\label{32}
\sum_iR_i\equiv p^{i_0}N~(\operatorname{mod}p^{r!}-1).
\end{equation}
To prove (i), it remains to find an $m_s'\in\mathbb{N}^*$ such that $(p^{r!}-1)|m_s'$. Thanks to (\ref{31}), we can find $0\le n_j\le m_{sj}$ for each $j\equiv i_0$ $(\operatorname{mod}r!)$, such that $\sum_{j\equiv i_0 ({\rm mod}r!)}n_j=N$. We denote
$$
m_{sj}'=\left\{
\begin{array}{ll}
m_{sj}-n_j & j\equiv i_0 (\operatorname{mod}r!)\\
m_{sj} & \mbox{otherwise}\\
\end{array}
\right..
$$
Let $m_s'=\sum_jm_{sj}'p^j$ and $m_s'=\sum_tR_t'q^t$ be its $q$-adic expansion. Since
 $$\sum_{j\equiv i_0 ({\rm mod}r!)}m_{sj}\ge p^{r!}-1>N=\sum_{j\equiv i_0 ({\rm mod}r!)}n_j,$$
 it follows that $m_{sj}'=m_{sj}-n_j>0$ for some $j\equiv i_0~({\rm mod}r!)$, which implies that $m_s'\ne0$.
 It is clear that $$\sum_iR_i'=\sum_iR_i-p^{i_0}N\equiv 0~(\operatorname{mod}p^{r!}-1)$$
by (\ref{32}), which implies that $(p^{r!}-1)|m_s'$. Moreover, since $m_{sj}'\le m_{sj}$ by definition, we have $\displaystyle\left({m_s}\atop m_s'\right)\neq0$ $(\operatorname{mod}p)$ by Lemma \ref{Binomial} (i). Thus, (i) is proved.
\end{proof}

\begin{Lem}\label{red}
The following conditions on $\theta=(m_n)_{n\in\mathbb{N}^*}\in\widehat{\Bbbk^*}$ are equivalent:

\noindent$(\operatorname{i})$ $f(m_n)$ $(n\in\mathbb{N}^*)$ are bounded;

\noindent$(\operatorname{ii})$ $\theta$ is trivial, or there exist positive integers $N$, $0<\theta_1,\cdots,\theta_l<p$ and $g_{i,n}\in\mathbb{N}$ $(1\le i\le l, n\ge N)$ satisfying $0\le g_{i,n}<n!$, $g_{i,n+1}\equiv g_{i,n}$ $(\operatorname{mod}n!)$ for all $1\le i\le l$ and $g_{i,n}\ne g_{j,n}$ if $i\ne j$, such that
$m_n=\sum_{1\le i\le l}\theta_ip^{g_{i,n}}$ for all $n\ge N$.
\end{Lem}
\begin{proof}
(ii) $\Rightarrow$ (i): If $\theta$ is trivial, then $f(m_n)=0$ for all $n$. Otherwise, (ii) implies that $f(m_n)\le\max\{f(m_1),\cdots,f(m_N),\sum_{1\le i\le l}\theta_i\}$, and hence (i) holds.

\smallskip
\noindent (i) $\Rightarrow$ (ii): Assume that $\theta$ is nontrivial. Lemma \ref{cor33} implies that $\{f(m_n)\}_{n\in\mathbb{N}}$ is non-decreasing. Due to (i), there exists $N_1\in\mathbb{N}$ such that
\begin{equation}\label{34}
f(m_{N_1})=f(m_{N_1+1})=f(m_{N_1+2})=\cdots.
\end{equation}
On the other hand, (i) implies that $\{M_{m_n}\}_{n\in\mathbb{N}}$ is bounded. Combining this, (\ref{34}) and Lemma \ref{cor33}, we see that $0\ne M_{m_N}=M_{m_{N+1}}=M_{m_{N+2}}=\cdots$
for some $N\ge N_1$.

 Denote $m_N=\sum_{i=1}^l\theta_ip^{g_{i,N}}$ (with all $\theta_i\ne0$) the $p$-adic expansion of $m_N$. In particular, $g_{i,N}$ $(1\le i\le l)$ are distinct. Suppose that $g_{i,n}$ $(1\le i\le l)$ satisfying requirements in (ii) are determined for $N\le n\le k$, we will determine $g_{i,k+1}$ $(1\le i\le l)$ satisfying requirements in (ii). Let $m_{k+1}=\sum_jm_{k+1,j}p^j$ be the $p$-adic expansion of $m_{k+1}$ and $r_i'=\sum_{j\equiv i(\operatorname{mod}k!)}m_{k+1,j}$ ($1\le i< k!$). Since $f(m_{k+1})=f(m_k)$, we have
 \begin{equation}\label{pikachu}
 r_{g_{i,k}}'=\theta_i~(1\le i\le l)~{\rm and} ~r_t'=0~{\rm if}~t\ne g_{i,k}~(1\le i\le l)
 \end{equation}
 by Lemma \ref{cor33}. Combining $M_{m_{k+1}}=M_{m_k}$ and (\ref{pikachu}), we see that $m_{k+1,j}=0$ if $j\not\equiv g_{i,k}$ $(\op{mod}k!)$ $(1\le i\le l)$, and for each $1\le i\le l$, there is a unique $g_{i,k+1}\in\mathbb{N}$ satisfying $g_{i,k+1}\equiv g_{i,k}$ $(\op{mod}k!)$, such that $m_{k+1,g_{i,k+1}}=\theta_i$ and $m_{k+1,j}=0$ if $j\equiv g_{i,k}$ $(\op{mod}k!)$ and $j\ne g_{i,k+1}$. Since $g_{i,k}~(1\le i\le l)$ are distinct by assumption, it follows that $g_{i,k+1}~(1\le i\le l)$ are distinct. To summarize, we have  $m_{k+1}=\sum_{1\le i\le l}\theta_ip^{g_{i,k+1}}$. This completes the proof.
\end{proof}

\begin{Rem}\label{Galois}
\normalfont
Let $\mathbb{G}=\operatorname{Gal}(\bar{\mathbb{F}}_p/\mathbb{F}_p)$, the absolute Galois group of $\mathbb{F}_p$. It is known that $\mathbb{G}$ is isomorphic to the inverse limit of the system $\cdots\rightarrow\mathbb{Z}/(n+1)!\mathbb{Z}\rightarrow\mathbb{Z}/n!\mathbb{Z}\rightarrow\cdots$.
Therefore, the congruence relation in Lemma \ref{red} (ii) is equivalent to say that the sequence $(g_{i,n})_{n\ge N}$ corresponds to an automorphism $\omega_i\in\mathbb{G}$.
\end{Rem}

Let $\mathcal{X}_0$ (resp. $\mathcal{X}_1$) be the subset of $\widehat{\Bbbk^*}$ satisfying the equivalent conditions in Lemma \ref{red} (resp. Lemma \ref{padic}). Then $\widehat{\Bbbk^*}=\mathcal{X}_0\cup\mathcal{X}_1$.
\begin{Ex}
\normalfont
In the case $G=SL_2(\Bbbk)$, it is clear that the dominant weights are in $\mathcal{X}_0$. The anti-dominant weights (i.e. the rational characters corresponding to $\mathbb{Z}_{<0}$) are in $\mathcal{X}_1$. Indeed, if $\theta=(m_n)_{n\in\mathbb{N}^*}$ is an anti-dominant weight, then there is a $\lambda\in\mathbb{N}^*$ such that $m_n=p^{n!}-1-\lambda$ if $n$ is large enough, and hence $\{f(m_n)\}_{n\in\mathbb{N}^*}$ is unbounded.
\end{Ex}

It is known that ${\bf T}_i$ (defined below Assumption 1.1) is isomorphic to the standard torus of $SL_2(\Bbbk)$ (which is isomorphic to $\Bbbk^*$) or $PGL_2(\Bbbk)$. If $\mathcal{D}{\bf G}$, the derived group of ${\bf G}$, is simply connected and semisimple, then ${\bf T}_i\simeq\Bbbk^*$ $(1\le i\le |I|)$. Let $\mathbb{X}=\{\theta\in\widehat{\bf T}\mid \theta|_{{\bf T}'}\in X({\bf T}')\}$ and $\mathbb{X}_1=\{\theta\in \mathbb{X}\mid0\le\langle\theta|_{{\bf T}'},\alpha_i^\vee\rangle<p~(i\in I)\}$.
\begin{Lem}\label{kkk}
Assume $\mathcal{D}{\bf G}$ is simply connected and semisimple and let $\theta\in\widehat{\bf T}$. If $\theta|_{{\bf T}_i}\in\mathcal{X}_0$ for any $i\in I$, then there exist $\theta_1,\cdots,\theta_l\in \mathbb{X}_1$ and distinct automorphisms $\omega_1,\cdots,\omega_l\in\mathbb{G}$ such that $\theta=\sum_{i=1}^l\theta_i^{\omega_i}$.
\end{Lem}
\begin{proof}
We proceed by 3 steps.

\smallskip
\noindent{\bf Step 1}: {\it Assume that ${\bf G}=SL_2(\Bbbk)$.}

\noindent In this case we have $\theta=(m_n)_{n\in\mathbb{N}^*}$. By Lemma \ref{red}, there are integers $0<\theta_1,\cdots,\theta_l<p$ and $N$ such that $m_n=\sum_{i=1}^l\theta_ip^{g_{i,n}}$ for all $n\ge N$, where all $g_{i,n}$ satisfy requirements in Lemma \ref{red} (ii). Equivalently, for any $n\ge N$, we have $\theta=\sum_{i=1}^l\theta_i^{[g_{i,n}]}$ when restricted to $\mathbb{F}_{p^{n!}}$, where $\theta_i$ is identified with the character $x\mapsto x^{\theta_i}$ and $[m]$ means the $m$-th Frobenius twist. For each $1\le i\le l$, let $(g_{i,n})_{n\ge N}$ corresponds to $\omega_i\in\mathbb{G}$ as in Remark \ref{Galois}. Thus, we have $\theta=\sum_{i=1}^l\theta_i^{\omega_i}$ when lifted to $\Bbbk^*$ as desired.

\smallskip
\noindent{\bf Step 2}: {\it Assume that ${\bf G}$ is simply connected and semisimple.}

\noindent By assumption, ${\bf T}$ is isomorphic to the direct product of ${\bf T}_i$ $(1\le i\le|I|)$. In this case we have the (unique) decomposition $\theta=\sum_{i=1}^{|I|}\lambda_i$ such that for each $i$, we have $\lambda_i|_{{\bf T}_j}=\op{tr}$ if $j\ne i$, and each $\lambda_i$ is identified with a character of ${\bf T}_i\simeq\Bbbk^*$. By Step 1, for each $1\le i\le |I|$, there exist $\lambda_{i,1},\cdots,\lambda_{i,l_i}\in \mathbb{X}_1$ and automorphisms $\tau_{i,1},\cdots,\tau_{i,l_i}\in\mathbb{G}$ such that (i) $\langle\lambda_{i,l_k},\alpha_j^\vee\rangle=0$ ($j\ne i,~1\le k\le l_i$); (ii) $\tau_{i,j}\ne\tau_{i,k}$ if $j\ne k$; (iii) $\lambda_i=\sum_{j=1}^{l_i}\lambda_{i,j}^{\tau_{i,j}}$.

\noindent It follows that
\begin{equation}\label{theta}
\theta=\sum_{{1\le i\le |I|}\atop{1\le j\le l_i}}\lambda_{i,j}^{\tau_{i,j}}=\sum_{\tau\in\mathbb{G}}(\sum_{\tau_{i,j}=\tau}\lambda_{i,j})^\tau.
\end{equation}
It is clear that $\sum_{\tau_{i,j}=\tau}\lambda_{i,j}\ne0$ for finitely many $\tau$ (we denote $\omega_1,\cdots,\omega_l$ for them). Write $\theta_k=\sum_{\tau_{i,j}=\omega_k}\lambda_{i,j}$. For each $t\in I$, (i) implies that $\langle\lambda_{i,j},\alpha_t^\vee\rangle=0$ if $i\ne t$, and (ii) implies that for each $t\in I$, there is at most one $1\le j_0\le l_t$ such that $\tau_{t,j_0}=\omega_k$, and hence $$\langle\theta_k,\alpha_t^\vee\rangle=\langle\sum_{\tau_{i,j}=\omega_k}\lambda_{i,j},\alpha_t^\vee\rangle
=\langle\sum_{\tau_{t,j}=\omega_k}\lambda_{t,j},\alpha_t^\vee\rangle\le\langle\lambda_{t,j_0},\alpha_t^\vee\rangle<p,$$ which implies that $\theta_k\in \mathbb{X}_1$ ($1\le k\le l$). Thus, (\ref{theta}) becomes $\theta=\sum_{k=1}^l\theta_k^{\omega_k}$.

\smallskip
\noindent{\bf Step 3}: {\it General case}.

\noindent Let $\pi:\widehat{\bf T}\rightarrow\widehat{{\bf T}'}$ be the restriction (cf. Section 1). Since $\pi$ commutes with $F$-action and $\mathcal{D}{\bf G}$ is simply connected and semisimple, Step 2 and Lemma \ref{surj} imply that $\pi(\theta)=\sum_{i=1}^l\pi(\mu_i)^{\omega_i}=\sum_{i=1}^l\pi(\mu_i^{\omega_i})$ for some $\mu_1,\cdots,\mu_l\in \mathbb{X}_1$ and distinct automorphisms $\omega_1,\cdots,\omega_l\in\mathbb{G}$. Lifting to $\widehat{\bf T}$, one obtains $\theta=\mu+\sum_{i=1}^l\mu_i^{\omega_i}$ for some $\mu=\operatorname{Ker}\pi$. Therefore, $\theta_1=\mu^{\omega_1^{-1}}+\mu_1$ and $\theta_i=\mu_i$ $(2\le i\le l)$ satisfy the requirement. This completes the proof.
\end{proof}

\section{Classification of simple modules with ${\bf B}$-stable line}
In this section, we will give the classification of irreducible $\Bbbk{\bf G}$-modules with ${\bf B}$-stable Line up to isomorphism. We will need the following result on the head and socle of the $\Bbbk G_n$-module $\Bbbk G_n{\bf 1}_\theta$.
\begin{Lem}\cite[4.6 and 6.1]{Jan2}\label{Soc}
Assume that ${\bf G}$ is semisimple. Let $\theta\in\widehat{\bf T}$ and $\lambda=\theta|_{T_n}$, and $\theta|_{T_{i,n}}\ne\op{tr}$ for any $i\in I$. Then

\noindent$(\op{i})$ $\Bbbk G_n{\bf 1}_\theta$ has simple head and socle;

\noindent$(\op{ii})$ There is a surjection $\Bbbk G_n{\bf 1}_\theta\rightarrow L(\lambda)$ sending ${\bf 1}_\theta$ to $v_\lambda$, the highest weight vector of $L(\lambda)$;

\noindent$(\op{iii})$ $\op{Soc}\Bbbk G_n{\bf 1}_\theta=\Bbbk G_n\underline{U_n}\dot{w_0}{\bf 1}_\theta$, where $w_0$ is the longest element in $W$.
\end{Lem}

\noindent We start with the rank 1 case.

\begin{Lem}\label{sl2case}
Let ${\bf G}=SL_2(\Bbbk)$ or $PGL_2(\Bbbk)$ and $\theta=(m_n)_{n\in\mathbb{N}^*}\in\widehat{\bf T}$ $($In the $PGL_2$ case, one identifies $\theta$ with its pull back via surjection $\Bbbk^*\rightarrow{\bf T}$$)$. Then $\mathbb{M}(\theta)$ is irreducible if and only if $\theta\in\mathcal{X}_1$.
\end{Lem}
\begin{proof}
Since $\mathbb{M}({\rm tr})$ is reducible (it contains infinite-dimensional Steinberg module as a proper submodule by \cite[Proposition 2.3]{Xi}). So it is enough to prove the result for $\theta\ne{\rm tr}$. We will prove this by showing that if $\theta\ne{\rm tr}$, then the following statements are equivalent:

(1) $\mathbb{M}(\theta)$ is irreducible;

(2) $\mathbb{M}(\theta)=\Bbbk{\bf G}\underline{U_r}s{\bf 1}_\theta$ for any $r\in\mathbb{N}^*$;

(3) For any $r\in\mathbb{N}^*$, there is a $t>r$, such that $\Bbbk G_t{\bf 1}_\theta=\Bbbk G_t\underline{U_r}s{\bf 1}_\theta$;

(4) $\theta\in\mathcal{X}_1$.

\medskip
(1) $\Rightarrow$ (2) is trivial.

(2) $\Rightarrow$ (1): Since $\theta\ne{\rm tr}$, there is an $N$ such that $m_{n}>0$ for all $n>N$. For such $n$, we have (as $\Bbbk G_n$-modules) $\op{Soc}\Bbbk G_n{\bf 1}_\theta$ is simple and equals to $\Bbbk G_n\underline{U_n}s{\bf 1}_\theta$ by Lemma \ref{Soc}. For any $x\in\mathbb{M}(\theta)$, we have $x\in\Bbbk G_r{\bf 1}_\theta$ for some $r>N$ since $\mathbb{M}(\theta)=\bigcup_{i>0}\Bbbk G_i{\bf 1}_\theta$. One has $\op{Soc}\Bbbk G_rx=\op{Soc}\Bbbk G_r{\bf 1}_\theta=\Bbbk G_r\underline{U_r}s{\bf 1}_\theta$ by the above discussion. Therefore, we have
$$\Bbbk{\bf G}x=\Bbbk{\bf G}(\Bbbk G_rx)\supset\Bbbk{\bf G}(\op{Soc}\Bbbk G_rx)=\Bbbk{\bf G}\underline{U_r}s{\bf 1}_\theta=\mathbb{M}(\theta),$$
which forces $\mathbb{M}(\theta)=\Bbbk{\bf G}x$, and hence $\mathbb{M}(\theta)$ is irreducible.

(2) $\Rightarrow$ (3): We have ${\bf 1}_\theta\in\Bbbk{\bf G}\underline{U_r}s{\bf 1}_\theta=\bigcup_{t>r}\Bbbk G_t\underline{U_r}s{\bf 1}_\theta$ by (2). It follows that ${\bf 1}_\theta\in\Bbbk G_t\underline{U_r}s{\bf 1}_\theta$ for some $t>r$, and hence $\Bbbk G_t{\bf 1}_\theta\subset\Bbbk G_t\underline{U_r}s{\bf 1}_\theta$, which forces $\Bbbk G_t{\bf 1}_\theta=\Bbbk G_t\underline{U_r}s{\bf 1}_\theta$.

(3) $\Rightarrow$ (2): Follows from
$$\mathbb{M}(\theta)=\Bbbk{\bf G}{\bf 1}_\theta=\Bbbk{\bf G}(\Bbbk G_t{\bf 1}_\theta)=\Bbbk{\bf G}(\Bbbk G_t\underline{U_r}s{\bf 1}_\theta)=\Bbbk{\bf G}\underline{U_r}s{\bf 1}_\theta.$$

(3) $\Leftrightarrow$ (4): Let $\pi$ be the composition map
$\Bbbk G_t{\bf 1}_\theta\rightarrow L(m_t)\rightarrow \nabla(m_t)$,
where the first is the natural projection sending ${\bf 1}_\theta$ to $v_0$ (recall $v_i$'s in the end of Subsection 1.2), and the second is inclusion. We claim that
\begin{center}(3) holds if and only if $\pi(\underline{U_r}s{\bf 1}_\theta)\neq0$.\end{center} The ``only if" part is clear. It remains to prove the ``if" part.

If $m_t=0$, then $\Bbbk G_t{\bf 1}_\theta$ is a direct sum of two irreducible $\Bbbk G_t$-modules, namely, $\Bbbk G_t{\bf 1}_\theta=\Bbbk({\bf 1}_\theta+\underline{U_t}s{\bf 1}_\theta)\oplus\Bbbk G_t(1-s){\bf 1}_\theta$ as $\Bbbk G_t$--modules, in which the left factor is trivial module and the right one is Steinberg module. Now $\pi(\underline{U_r}s{\bf 1}_\theta)\neq0$ implies that the projection of $\underline{U_r}s{\bf 1}_\theta$ to the left factor is nonzero. Moreover, since $\Bbbk G_t\underline{U_r}s{\bf 1}_\theta$ is not a trivial $\Bbbk G_t$-module, the projection of $\underline{U_r}s{\bf 1}_\theta$ to the right factor is also nonzero, and hence $\Bbbk G_t\underline{U_r}s{\bf 1}_\theta=\Bbbk G_t{\bf 1}_\theta$.

If $m_t\ne0$, then $\operatorname{Hd}\Bbbk G_t{\bf 1}_\theta=L(m_t)$ by Lemma \ref{Soc}, equivalently, there is a unique maximal $\Bbbk G_t$-submodule $M$ of $\Bbbk G_t{\bf 1}_\theta$. If $\pi(\underline{U_r}s{\bf 1}_\theta)\neq0$, then $M'=\Bbbk G_t\underline{U_r}s{\bf 1}_\theta\not\subset M$. This forces $M'=\Bbbk G_t{\bf 1}_\theta$ since any proper $\Bbbk G_t$-submodule is contained in $M$. Thus, the claim is proved.

Let $v_i~(0\leq i\leq m_t)$ be the basis of $\nabla(m_t)$ satisfying (\ref{H0}). Then
\begin{equation}\label{pi}
\pi(\underline{U_r}s{\bf 1}_\theta)=\sum_{a\in\mathbb{F}_{p^{r!}}}\begin{pmatrix}1&a\\ 0&1\end{pmatrix}v_{m_t}=\sum_{{0\leq l\leq m_t}\atop{a\in\mathbb{F}_{p^{r!}}}}\left({m_t}\atop l\right)a^lv_{m_t-l}
\end{equation}
by (\ref{H0}). Combining (\ref{pi}) and Lemma \ref{Binomial} (ii), $\pi(\underline{U_r}s{\bf 1}_\theta)\neq0$ is equivalent to $\displaystyle\left({m_t}\atop{k(p^{r!}-1)}\right)\neq0$ $(\operatorname{mod}p)$ for some $k\in\mathbb{N}^*$, and this holds if and only if $\theta\in\mathcal{X}_1$ by Lemma \ref{padic}.
\end{proof}

\begin{Rem}\label{rem1}
\normalfont Actually, the proof of Lemma \ref{sl2case} shows that if $\theta\in\mathcal{X}_0$, then for any $a\in\mathbb{N}$, there exists $b>a$ such that $\Bbbk G_b{\bf 1}_\theta=\Bbbk G_b\underline{U_a}s{\bf 1}_\theta$.
\end{Rem}

\noindent Lemma \ref{key}-\ref{ce} below are technical results used to prove Theorem \ref{hd=sim} and \ref{paraind} below.
\begin{Lem}\label{key}
Let $w\in W$ and $A=\{\alpha_1,\alpha_2,\dots, \alpha_m\}$ and $B=\{\beta_1,\beta_2,\dots, \beta_n\}$ be two disjoint subsets of $\Phi_{w^{-1}}^-$ $($recall its definition in 1.1$)$, and assume that $(\sum_{i=1}^m\mathbb{N}\alpha_i)\cap\Phi\subset A$.
Assume that $b>a$, and denote
$$\delta:=\underline{U_{\alpha_1,b}}\cdots\underline{U_{\alpha_m,b}}\cdot
\underline{U_{\beta_1,a}}\cdots\underline{U_{\beta_n,a}}\dot{w}{\bf 1}_\theta.$$

\noindent We have

\noindent$\operatorname{(i)}$ Assume that $(\mathbb{N}^*\beta_1+\mathbb{N}^*\alpha_i)\cap\Phi\subset A$ for all $1\le i\le m$. Then
$$x\delta=\underline{U_{\alpha_1,b}}\cdots\underline{U_{\alpha_m,b}}\cdot
x \underline{U_{\beta_1,a}}\cdots\underline{U_{\beta_n,a}}\dot{w}{\bf 1}_\theta$$
for any $x\in U_{\beta_1,b}$.

\noindent$\operatorname{(ii)}$ Let $\gamma\in\Phi_{w^{-1}}^+$ and assume that $$((\mathbb{N}^*\gamma+\sum_{i=1}^m\mathbb{N}\alpha_i+\sum_{i=1}^n\mathbb{N}\beta_i)\backslash\{\gamma\})\cap\Phi_{w^{-1}}^-\subset A.$$ Then $y\delta=\delta$ for any $y\in U_{\gamma,b}$.
\end{Lem}
\begin{proof}
The proof is similar to that of \cite[Lemma 4.5]{CD2}. We rewrite the proof for the convenience of readers. It is clear that $U_{A,b}=\prod_{i=1}^mU_{\alpha_i,b}$ is a group since $(\sum_{i=1}^m\mathbb{N}\alpha_i)\cap\Phi\subset A$.

\medskip
\noindent$\op{(i)}$ By (\ref{commrel}) and assumption, we have $xu_ix^{-1}u_i^{-1}\in U_{A,b}$ (and hence $xu_ix^{-1}\in U_{A,b}$) for all $1\le i\le m$ and $u_i\in U_{\alpha_i,b}$. Thus, we have $xU_{A,b}x^{-1}=U_{A,b}$ and hence $x$ commutes with $\underline{U_{A,b}}=\underline{U_{\alpha_1,b}}\cdots\underline{U_{\alpha_m,b}}$ which proves (i).

\medskip

\noindent $\op{(ii)}$ By assumption and applying (\ref{commrel}) repeatedly, for any $y\in U_{\gamma,b}$, $g_1\in U_{A,b}$, and $g_2\in U_{B,a}=\prod_{1\leq i\leq n}U_{\beta_i,a}$, we have
\begin{equation}\label{==}
yg_1g_2=\sigma(g_1)g_2z,
\end{equation}
where $\sigma(g_1)\in U_{A,b}$ and $z\in{\bf U}_{w^{-1}}'$. We claim that the map $g_1\mapsto\sigma(g_1)$ (for fixed $y$ and $g_2$) is injective. Indeed, assume that $\sigma(g_1)=\sigma(g_1')$ for some $g_1'\in U_{A,b}$. Since $yg_1'g_2=\sigma(g_1)g_2z'$ for some $z'\in{\bf U}_{w^{-1}}'$, we have
\begin{equation}\label{=}
g_2^{-1}g_1^{-1}g_1'g_2=z^{-1}z'.
\end{equation}
It follows from equation (\ref{=}) that $g_2^{-1}g_1^{-1}g_1'g_2\in {\bf U}_{w^{-1}}\cap{\bf U}_{w^{-1}}'=\{1\}$, and hence $g_1=g_1'$ which proves the claim. Since $z\dot{w}{\bf 1}_\theta=\dot{w}{\bf 1}_\theta$ for any $z\in{\bf U}_{w^{-1}}'$, we have $y\delta=\delta$ for any $y\in U_{\gamma,b}$ thanks to equation (\ref{==}) and the injectivity of $\sigma$.
\end{proof}

For each $\theta\in\widehat{\bf T}$, define
$$I(\theta):=\{i\in I\mid\theta|_{{\bf T}_i}=\op{tr}\}.$$
\begin{Lem}\label{forsomew}
Let $a\in\mathbb{N}^*$, $\theta\in\widehat{\bf T}$ and $J\subset I(\theta)$, and let $\xi=\sum_{w\in W_J}c_w\underline{U_{w^{-1},a}}\dot{w}{\bf 1}_\theta\ne0$ $(c_w\in\Bbbk)$. Let $Y_\xi=\{w\in W_J|c_w\ne0\}$. Then
$$\Bbbk{\bf G}\xi=\Bbbk{\bf G}\underline{U_{w'^{-1},c}}\dot{w'}{\bf 1}_\theta$$
for some $w'\in Y_\xi$ and $c>a$.
\end{Lem}

\begin{proof}
The proof is similar to that of \cite[Lemma 4.7]{CD3}. We give the proof for the convenience of readers.

Let $M=\Bbbk{\bf G}\xi$ and $\Phi_{\xi}=\bigcup_{w\in Y_\xi}\Phi_{w^{-1}}^-$. We fix an order in $\Phi_\xi$ such that $\Phi_\xi=\{\beta_1,\cdots,\beta_m\}$ with $\op{ht}(\beta_1)\geq\cdots\geq\op{ht}(\beta_m)$.

Let $b>a$. For each $w\in Y_\xi$, write $\Phi_{w^{-1}}^-=\{\gamma_1,\cdots,\gamma_t\}$ with the order inherited from $\Phi_\xi$ (In particular, $\op{ht}(\gamma_1)\geq\cdots\geq\op{ht}(\gamma_t)$). For any $0\leq d\leq t$, set
$$\Theta(w,d,b,a):=\underline{U_{\gamma_1,b}}\cdots\underline{U_{\gamma_d,b}}\cdot
\underline{U_{\gamma_{d+1},a}}\cdots\underline{U_{\gamma_t,a}},$$
if $d>0$ and
$$\Theta(w,0,b,a):=\underline{U_{\gamma_1,a}}\cdots\underline{U_{\gamma_t,a}}.$$
The result follows from the following claim (applied to $Y=Y_\xi$ and $d=0$) whose proof is similar to that of \cite[Proposition 4.3]{CD2}.

\smallskip
\noindent{\it Claim: }{\it Let $Y$ be a nonempty subset of $Y_{\xi}$ and $\Phi_Y=\bigcup_{w\in Y}\Phi_{w^{-1}}^-=\{\alpha_1,\cdots,\alpha_n\}$ with the order inherited from $\Phi_\xi$ $($In particular $\op{ht}(\alpha_1)\geq\cdots\geq\op{ht}(\alpha_n)$$)$, and let $d\geq 0$ be an integer such that $\alpha_1,\dots,\alpha_d\in\bigcap_{w\in Y}\Phi_{w^{-1}}^-$. Assume that $b>a$ and $$\xi_d:=\sum_{w\in Y}c_w\Theta(w,d,b,a)\dot{w}{\bf 1}_\theta\in M.$$ Then $\underline{U_{w^{-1},b}}\dot{w}{\bf 1}_\theta\in M$ for some $w\in Y$}.

\noindent{\it Proof of the claim:} We will prove this claim by the induction on $|Y|$.
If $|Y|=1$, then $\xi_d=c\Theta(w,d,b,a)\dot{w}{\bf 1}_\theta\in M$ for some $c\in\Bbbk^*$ and $w\in Y$. We consider the $\Bbbk U_{b}$-module $N=\Bbbk U_{b} \Theta(w,d,b,a)\dot{w}{\bf 1}_\theta\subset M$. Clearly, $N^{U_{b}}\neq0$ by Lemma \ref{fixedpt}. Note that $N^{U_{b}}\subseteq(\Bbbk U_{b}\dot{w}{\bf 1}_\theta)^{U_{b}}= \Bbbk\underline{U_{w^{-1},b}}\dot{w}{\bf 1}_\theta$, then $\underline{U_{w^{-1},b}}\dot{w}{\bf 1}_\theta\in M$.

\smallskip
Assume that $|Y|>1$. Let $I_i$ be a set of left coset representatives of $U_{\alpha_i,a}$ in $U_{\alpha_i,b}$. Let $l$ be the minimal number such that $\alpha_{d+l}\not\in\Phi_{w^{-1}}^-$ for some $w\in Y$. Since $\Phi_{w_1}^-\neq\Phi_{w_2}^-$ if $w_1\neq w_2$, such $l$ always exists.

If $w\in Y$ and $\alpha_{d+l}\not\in\Phi_{w^{-1}}^-$, by our assumption on the order in each $\Phi_{w^{-1}}^-$ and applying Lemma \ref{key} (i) to $A=\{\gamma_1,\cdots,\gamma_{d+i}\}$, $B=\{\gamma_{d+i+1},\cdots,\gamma_t\}$ and $\beta=\gamma_{d+i+1}$ yields
\begin{equation}\label{d+i+1}
\underline{I_{d+i+1}}\Theta(w,d+i,b,a)\dot{w}{\bf 1}_\theta
=\Theta(w,d+i+1,b,a)\dot{w}{\bf 1}_\theta
\end{equation}
for all $0\le i<l-1$, and applying
Lemma \ref{key} (ii) to $A=\{\gamma_1,\cdots,\gamma_{d+l-1}\}$, $B=\{\gamma_{d+l},\cdots,\gamma_t\}$ and $\gamma=\gamma_{d+l}$ yields
\begin{equation}\label{d+l}
\underline{I_{d+l}}\Theta(w,d+l-1,b,a)\dot{w}{\bf 1}_\theta=q^{b-a}\Theta(w,d+l-1,b,a)\dot{w}{\bf 1}_\theta=0
\end{equation}
since $\op{char}\Bbbk=p$ and $b\neq a$.
Thus, combining (\ref{d+i+1}) and (\ref{d+l}) yields
\begin{equation}\label{=0}
\underline{I_{d+l}}\cdots\underline{I_{d+1}}\Theta(w,d,b,a)\dot{w}{\bf 1}_\theta=0.
\end{equation}

If $w\in Y$ and $\alpha_{d+l}\in\Phi_{w^{-1}}^-$, we have
\begin{equation}\label{!=0}
\underline{I_{d+l}}\cdots\underline{I_{d+1}}\Theta(w,d,b,a)\dot{w}{\bf 1}_\theta
=\Theta(w,d+l,b,a)\dot{w}{\bf 1}_\theta
\end{equation}
by Lemma \ref{key} (i). Denote $\xi_{d+l}:=\underline{I_{d+l}}\cdots\underline{I_{d+1}}\xi_d\in M$ and let $Y'$ be the set of $w\in Y$ such that the coefficient of $\Theta(w,d+l,b,a)\dot{w}{\bf 1}_\theta$ in $\xi_{d+l}$ is nonzero. Combining (\ref{=0}), (\ref{!=0}), and the minimality of $l$, we see that $\xi_{d+l}\neq 0$ (equivalently, $Y'$ is nonempty) and  $Y'\subsetneq Y$ (In particular $|Y'|<|Y|$). Notice that $\{\alpha_1,\cdots,\alpha_{d+l}\}\subset\bigcap_{w\in Y'}\Phi_{w^{-1}}^-$, The claim follows from applying the induction hypothesis to $Y'$ and $\xi_{d+l}$.
\end{proof}

\begin{Rem}\label{wnee}
\normalfont
Keeping the notation as in Lemma \ref{forsomew}, then the proof of Lemma \ref{forsomew} indicates that if $c_e\ne0$ and $c_w\ne0$ for some $w\ne e$, then $w'\ne e$.
\end{Rem}

\noindent For any $\theta\in\widehat{\bf T}$ and $w\in W$, define $\theta^w\in\widehat{\bf T}$ by
$\theta^w(t)=\theta(\dot{w}^{-1}t\dot{w})$ for any $t\in{\bf T}$.
\begin{Lem}\label{c2}
Let $k\in I$, $w\in W$ and $\theta\in\widehat{\bf T}$ and assume that $s_kw>w$. Suppose that one of the following holds:

\noindent$(\operatorname{i})$ $({\bf U}_{w^{-1}})^{s_k}\ne{\bf U}_{w^{-1}}$;

\noindent$(\operatorname{ii})$ $({\bf U}_{w^{-1}})^{s_k}={\bf U}_{w^{-1}}$ and $\theta^w|_{{\bf T}_k}\in\mathcal{X}_1$,

\noindent where the superscript $s_k$ means conjugation. Then for any $a\in\mathbb{N}$, we have
$$\Bbbk{\bf G}\underline{U_{w^{-1},b}}\dot{w}{\bf 1}_\theta=\Bbbk{\bf G}\underline{U_{w^{-1}s_k,a}}\dot{s_k}\dot{w}{\bf 1}_\theta$$
for some $b>a$.
\end{Lem}
\begin{proof}
The proof is similar to \cite[Lemma 4.8]{CD3}. Let $M=\Bbbk{\bf G}\underline{U_{w^{-1}s_k,a}}s_k\dot{w}{\bf 1}_\theta$.

\medskip
\noindent If (i) holds, then $s_k\Phi_{w^{-1}}^-\backslash\Phi_{w^{-1}}^-\neq\emptyset$. Choose $\gamma\in s_k\Phi_{w^{-1}}^-\backslash\Phi_{w^{-1}}^-$ such that
$\operatorname{ht}(\gamma)$ is maximal.
Let $\Gamma=\{\alpha\in s_k\Phi_{w^{-1}}^-\cap\Phi_{w^{-1}}^-\mid\operatorname{ht}(\alpha)\ge\operatorname{ht}(\gamma)\}$
and $\Gamma'=\Phi_{w^{-1}}^-\backslash\Gamma$. Notice that
$$w^{-1}s_k(\beta)=w^{-1}(\beta)-\langle\beta,\alpha_k^\vee\rangle w^{-1}(\alpha_k)\in\Phi^+$$
for any $\beta\in\Phi_{w^{-1}}^-\backslash s_k\Phi_{w^{-1}}^-$ which forces $\langle\beta,\alpha_k^\vee\rangle<0$, and hence $s_k(\beta)>\beta$. It follows that $\operatorname{ht}(\gamma)>\operatorname{ht}(\beta)$ for any $\beta\in\Phi_{w^{-1}}^-\backslash s_k\Phi_{w^{-1}}^-$. So $\gamma$, $A=\Gamma$, $B=\Gamma'$ satisfy the assumption in Lemma \ref{key} (ii). It is clear that ${\bf U}_\Gamma=\prod_{\alpha\in\Gamma}{\bf U}_\alpha$ is a normal subgroup of ${\bf U}$ by the definition of $\Gamma$ and (\ref{commrel}). Choose a $b>a$ and let $C$ (resp. $D$) be a complete set of representatives of left cosets of $U_{\Gamma,a}$ (resp. $U_{\gamma,a}$) in $U_{\Gamma,b}$ (resp. $U_{\gamma,b}$). Therefore, applying Lemma \ref{key} (ii) to $\gamma$, $A=\Gamma$, $B=\Gamma'$ yields
\begin{equation}\label{biu1}
\underline{D}\cdot\underline{C}\cdot\underline{U_{w^{-1},a}}\dot{w}{\bf 1}_\theta=\underline{D}\cdot\underline{U_{\Gamma,b}}\prod_{\alpha\in\Gamma'}\underline{U_{\alpha,a}}\dot{w}{\bf 1}_\theta=p^{b-a}\underline{U_{\Gamma,b}}\prod_{\alpha\in\Gamma'}\underline{U_{\alpha,a}}\dot{w}{\bf 1}_\theta=0.
\end{equation}
Here, the product is taken with respect to a fixed order in $\Gamma'$.
It follows from (\ref{sus}) that
\begin{equation}\label{biu4}
\dot{s_k}^{-1}\underline{U_{w^{-1}s_k,a}}\dot{s_k}\dot{w}{\bf 1}_\theta=\underline{U_{w^{-1},a}}\dot{w}{\bf 1}_\theta+\xi.
\end{equation}
and
\begin{equation}\label{biu5}
\dot{s_k}^{-1}\underline{U_{\alpha_k,a}}\cdot\underline{(U_{w^{-1},b})^{s_k}}\dot{s_k}\dot{w}{\bf 1}_\theta=\underline{U_{w^{-1},b}}\dot{w}{\bf 1}_\theta+\xi',
\end{equation}
where
$$\xi=\sum_{t\in\mathbb{F}_{p^{a!}}^*}\theta^w
(h_k(-t^{-1}))\varepsilon_{\alpha_k}(t)\underline{(U_{w^{-1},a})^{s_k}}\dot{s_k}\dot{w}{\bf 1}_\theta$$
and
$$\xi':=\sum_{t\in\mathbb{F}_{p^{a!}}^*}\theta^w
(h_k(-t^{-1}))\varepsilon_{\alpha_k}(t)\underline{(U_{w^{-1},b})^{s_k}}\dot{s_k}\dot{w}{\bf 1}_\theta.$$
The conjugation by $\varepsilon_{\alpha_k}(t)$ takes $C$ to another complete set of left cosets of $U_{\Gamma,a}$ in $U_{\Gamma,b}$ by the normality of $U_{\Gamma,b}$, and $y^{-1}\varepsilon_{\alpha_k}(t)^{-1}y\varepsilon_{\alpha_k}(t)\in U_{\Gamma,b}$ for any $y\in D$ by (\ref{commrel}). Combining these and applying Lemma \ref{key} (i) to $\gamma$, $A=\Gamma$, $B=\Phi_{w^{-1}s_k}^-\backslash\Gamma$ yields
$$
\underline{D}\cdot\underline{C}\cdot\xi=
\sum_{t\in\mathbb{F}_{p^{a!}}^*}\theta^w
(h_k(-t^{-1}))\varepsilon_{\alpha_k}(t)\underline{U_{\Gamma,b}}\cdot\underline{U_{\gamma,b}}
\prod_{\alpha\in s_k\Phi_{w^{-1}}^-\backslash(\Gamma\cup\{\gamma\})}\underline{U_{\alpha,a}}\dot{s_k}\dot{w}{\bf 1}_\theta\ne0.
$$
Here, the product is taken with respect to a fixed height-downward order.

Since $(\mathbb{N}^*\alpha_k+\mathbb{N}^*\alpha)\cap\Phi\subset\Phi_{w^{-1}s_k}^-\backslash\{\alpha_k\}$ for any $\alpha\in\Phi_{w^{-1}s_k}^-\backslash\{\alpha_k\}$, one has $\varepsilon_{\alpha_k}(t)^{-1}(U_{w^{-1},b})^{s_k}\varepsilon_{\alpha_k}(t)=(U_{w^{-1},b})^{s_k}$ for any $t\in\mathbb{F}_{p^{a!}}^*$ by (\ref{commrel}), and hence the space $\varepsilon_{\alpha_k}(t)\cdot\Bbbk(U_{w^{-1},b})^{s_k}\dot{s_k}\dot{w}{\bf 1}_\theta$ is $(U_{w^{-1},b})^{s_k}$-invariant and $$(\varepsilon_{\alpha_k}(t)\cdot\Bbbk(U_{w^{-1},b})^{s_k}\dot{s_k}\dot{w}{\bf 1}_\theta)^{(U_{w^{-1},b})^{s_k}}=\Bbbk\varepsilon_{\alpha_k}(t)\underline{(U_{w^{-1},b})^{s_k}}\dot{s_k}\dot{w}{\bf 1}_\theta$$
for any $t\in\mathbb{F}_{p^{a!}}^*$.  It follows that
$$
\aligned
 &\ (\Bbbk (U_{w^{-1},b})^{s_k}\underline{D}\cdot\underline{C}\cdot\xi)^{(U_{w^{-1},b})^{s_k}}\\
=&\ \left(\sum_{t\in\mathbb{F}_{p^{a!}}}(\varepsilon_{\alpha_k}(t)\cdot\Bbbk(U_{w^{-1},b})^{s_k}\dot{s_k}\dot{w}{\bf 1}_\theta)^{(U_{w^{-1},b})^{s_k}}\right)\cap\Bbbk (U_{w^{-1},b})^{s_k}\underline{D}\cdot\underline{C}\cdot\xi\\
=&\ \left(\sum_{t\in\mathbb{F}_{p^{a!}}}\Bbbk\varepsilon_{\alpha_k}(t)\underline{(U_{w^{-1},b})^{s_k}}\dot{s_k}\dot{w}{\bf 1}_\theta\right)\cap\Bbbk (U_{w^{-1},b})^{s_k}\underline{D}\cdot\underline{C}\cdot\xi\\
=&\ \Bbbk\xi'.
\endaligned
$$
In particular, we have
\begin{equation}\label{biu3}
\xi'\in\Bbbk (U_{w^{-1},b})^{s_k}\underline{D}\cdot\underline{C}\cdot\xi.
\end{equation}

\noindent Combining (\ref{biu1}), (\ref{biu4}), and (\ref{biu3}), we see that
\begin{equation}\label{biu6}
\xi'\in\Bbbk (U_{w^{-1},b})^{s_k}\underline{D}\cdot\underline{C}\cdot \dot{s_k}^{-1}\underline{U_{w^{-1}s_k,a}}\dot{s_k}\dot{w}{\bf 1}_\theta\subset M.
\end{equation}
Clearly, left side of (\ref{biu5}) is in $M$, and hence $\underline{U_{w^{-1},b}}\dot{w}{\bf 1}_\theta\in M$ by (\ref{biu6}).

\medskip
\noindent If (ii) holds. It is clear that
\begin{equation}\label{v1}
\underline{U_{w^{-1}s_k,a}}\dot{s_k}\dot{w}{\bf 1}_\theta=(\underline{U_{w^{-1},a}})^{s_k}\cdot\underline{U_{\alpha_k,a}}\dot{s_k}\dot{w}{\bf 1}_\theta=\underline{U_{w^{-1},a}}\cdot\underline{U_{\alpha_k,a}}\dot{s_k}\dot{w}{\bf 1}_\theta.
\end{equation}
Since $\theta^w|_{{\bf T}_k}\in\mathcal{X}_1$,  there exists $b\in\mathbb{N}$ such that
\begin{equation}\label{v2}
\Bbbk G_{k,b}\dot{w}{\bf 1}_\theta=\Bbbk G_{k,b}\underline{U_{\alpha_k,a}}\dot{s_k}\dot{w}{\bf 1}_\theta
\end{equation}
by Remark \ref{rem1}. By multiplying the sum of representatives of all the left cosets of $U_{w^{-1},a}$ in $U_{w^{-1},b}$ to the right side of (\ref{v1}), one obtains
\begin{equation}\label{v3}
\underline{U_{w^{-1},b}}\cdot\underline{U_{\alpha_k,a}}\dot{s_k}\dot{w}{\bf 1}_\theta\in M.
\end{equation}
Meanwhile $U_{w^{-1},b}$ is invariant under $G_{k,b}$-conjugation, combining this and (\ref{v2}) yields
$$\underline{U_{w^{-1},b}}\dot{w}{\bf 1}_\theta\in M$$
which completes the proof.
\end{proof}

\begin{Lem}\label{ce}
 Assume that $c_w\in\Bbbk$ for each $w\in W$ and $c_e\ne0$. Then $$\mathbb{M}(\operatorname{tr})=\Bbbk{\bf G}\sum_{w\in W}c_w\underline{U_{w^{-1},a}}\dot{w}{\bf 1}_{\operatorname{tr}}$$
 for any $a\in\mathbb{N}$.
\end{Lem}
\begin{proof}
 For any $J\subset I$, define $\eta_J=\sum_{w\in W_J}(-1)^{\ell(w)}\dot{w}{\bf 1}_{\rm tr}\in\mathbb{M}({\rm tr})$ as in \cite{Xi}.
Write $v=\sum_{w\in W}c_w\underline{U_{w^{-1},a}}\dot{w}{\bf 1}_{\operatorname{tr}}$ and $N(v)=|\{w\in W|c_w\ne0\}|$. We will proceed by induction on $N(v)$. If $N(v)=1$, then  $c_w=0$ for all $w\ne e$ and the result is trivial. If $N(v)>1$, then $c_w\ne0$ for some $w\ne e$. Remark \ref{wnee} implies that
$$\underline{U_{w'^{-1},b}}\dot{w'}{\bf 1}_{\operatorname{tr}}\in\Bbbk{\bf G}\sum_{w\in W}c_w\underline{U_{w^{-1},a}}\dot{w}{\bf 1}_{\operatorname{tr}}$$
for some $w'\ne e$ and $b\in\mathbb{N}$. Let $K$ be the maximal subset of $I$ such that $w_K<w'$. Write $w'=s_{i_1}\cdots s_{i_t}w_K$ with $s_{i_m}s_{i_{m+1}}\cdots s_{i_t}w_K>s_{i_{m+1}}\cdots s_{i_t}w_K$ for all $1\le m<t$ and $s_{i_t}w_K>w_K$. The maximality of $K$ implies that $({\bf U}_{w_Ks_{i_t}\cdots s_{i_{m+1}}})^{s_{i_m}}\ne{\bf U}_{w_Ks_{i_t}\cdots s_{i_{m+1}}}$ for all $1\le m<t$. It follows that $\underline{U_{w_K,c}}\dot{w_K}{\bf 1}_{\operatorname{tr}}\in\Bbbk{\bf G}v$ for some $c\in\mathbb{N}$ by applying Lemma \ref{c2} (i) repeatedly, and hence
$$
\aligned
\eta_K&\ =\sum_{w\in W_K}p^{c!\ell(w)}\eta_K=\sum_{w\in W_K}(-1)^{\ell(w)}\dot{w}\underline{U_{w_K,c}}\eta_K\\
&\ =\sum_{w\in W_K}(-1)^{\ell(w)+\ell(w_K)}w\underline{U_{w_K,c}}\dot{w_K}{\bf 1}_{\operatorname{tr}}
\endaligned
$$
by \cite[Lemma 2]{St}, which implies that $\eta_K\in\Bbbk{\bf G}v$. Therefore, we obtain
\begin{equation}\label{hohoho}
\underline{U_{w'^{-1},a}}\dot{w'}{\bf 1}_{\operatorname{tr}}=\underline{U_{\tau^{-1}}}\dot{\tau}\underline{U_{w_K,a}}\dot{w_K}{\bf 1}_{\operatorname{tr}}=(-1)^{\ell(w_K)}\underline{U_{\tau^{-1}}}\dot{\tau}\underline{U_{w_K,a}}\eta_K\in\Bbbk{\bf G}v,
\end{equation}
where $\tau=s_{i_1}\cdots s_{i_t}$. Write $v'=v-c_{w'}\underline{U_{w'^{-1},a}}\dot{w'}{\bf 1}_{\operatorname{tr}}$, we see that $N(v')<N(v)$. It follows from (\ref{hohoho}) that $\mathbb{M}(\operatorname{tr})=\Bbbk{\bf G}v'\subset\Bbbk{\bf G}v$, where the first equality follows from induction. This forces $\mathbb{M}(\operatorname{tr})=\Bbbk{\bf G}v$ which completes the proof.
\end{proof}

Following \cite{CL} and \cite{YY}, for each $\theta\in\widehat{\bf T}$, $a\in\mathbb{N}$, and $w\in W_{I(\theta)}$, define $\mathcal{T}_{w,a}\in\mathbb{E}_{\theta,a}:=\operatorname{End}_{\Bbbk G_a}(\Bbbk G_a{\bf 1}_\theta)$ by $\mathcal{T}_{w,a}({\bf 1}_\theta)=\underline{U_{w,a}}w^{-1}{\bf 1}_\theta$. For each $J\subset I(\theta)$, set $$e_{J,a}=\sum_{w\in W_J}\mathcal{T}_{w,a},\quad o_{J,a}=(-1)^{\ell(w_J)}\mathcal{T}_{w_J,a}.$$

\begin{Prop}[{\cite[Proposition 4.5]{YY}}]\label{prop45}
Let $\theta\in\widehat{\bf T}$, $J\subset I(\theta)$, and $a\in\mathbb{N}$. Let $\{\pi_J|J\subset I(\theta)\}$ be a set of orthogonal primitive idempotents in $\mathbb{E}_{\theta,a}$ satisfying $1=\sum_{J\subset I(\theta)}\pi_J$ and $\pi_J\in e_{J,a}o_{\widehat{J},a}\mathbb{E}_{\theta,a}$, where $\widehat{J}=I(\theta)\backslash J$. Set $Y_a(J,\theta)=\operatorname{Im}\pi_J$. Then $\pi_J\mathbb{E}_{\theta,a}=e_{J,a}o_{\widehat{J},a}\mathbb{E}_{\theta,a}$, and we have an Krull--Schmidt decomposition $\Bbbk G_a{\bf 1}_\theta=\bigoplus_{J\subset I(\theta)}Y_a(J,\theta)$ with each $Y_a(J,\theta)$ having simple head and socle.
\end{Prop}

\begin{Rem}\label{317}
Applying Proposition \ref{prop45} to $J=I(\theta)$, we see that $\pi_{I(\theta)}\mathbb{E}_{\theta,a}=e_{I(\theta),a}\mathbb{E}_{\theta,a}$. In particular, we have
$$Y_a(I(\theta),\theta)=\operatorname{Im}(\pi_{I(\theta)})=\operatorname{Im}(e_{I(\theta),a})=\Bbbk G_a\sum_{w\in W_{I(\theta)}}\underline{U_{w^{-1},a}}\dot{w}{\bf 1}_\theta.$$
\end{Rem}

\begin{Thm}\label{hd=sim}
For any $\theta\in\widehat{\bf T}$, the $\Bbbk{\bf G}$-module $\mathbb{M}(\theta)$ has a unique simple quotient $($equivalently, maximal submodule$)$.
\end{Thm}
\begin{proof}
We will show that the sum of all proper submodules of $\mathbb{M}(\theta)$ is proper again, which implies $\mathbb{M}(\theta)$ has a unique maximal submodule.

For any $a\in\mathbb{N}$, let $h_a=e_{I(\theta),a}({\bf 1}_\theta)=\sum_{w\in W_{I(\theta)}}\underline{U_{w^{-1},a}}\dot{w}{\bf 1}_\theta$. Since ${\bf B}\cap{\bf G}_{I(\theta)}$ acts on ${\bf 1}_\theta$ trivially, applying Lemma \ref{ce} to ${\bf G}={\bf G}_{I(\theta)}$ yields $\Bbbk{\bf G}_{I(\theta)}{\bf 1}_\theta=\Bbbk{\bf G}_{I(\theta)}h_a$. In other words,
\begin{equation}\label{notinN}
h_a\not\in N~\mbox{for~any~proper}~\Bbbk{\bf G}\mbox{-submodule}~N~\mbox{of}~\mathbb{M}(\theta).
\end{equation}
Let $N_i$ $(i\in E)$ be all proper submodules of $\mathbb{M}(\theta)$. It remains to show that any finite sum $\sum_{i\in E}x_i$ with $x_i\in N_i$ is not equal to ${\bf 1}_\theta$. Since only finitely many $x_i$'s are nonzero, all nonzero $x_i$ are in $\Bbbk G_b{\bf 1}_\theta$ for some $b\in\mathbb{N}^*$. It follows from (\ref{notinN}) that $h_b\not\in\Bbbk{\bf G}x_i$ $(i\in E)$, and hence $h_b\not\in\Bbbk G_bx_i$ $(i\in E)$. Let $K_b$ be the kernel of the composite map $\Bbbk G_b{\bf 1}_\theta\rightarrow Y_b(I(\theta),\theta)\rightarrow\operatorname{Hd}Y_b(I(\theta),\theta)$, where the first map is projecting to direct summand $Y_b(I(\theta),\theta)$. By Proposition \ref{prop45} and Remark \ref{317}, $K_b$ is the unique maximal submodule of $\Bbbk G_b{\bf 1}_\theta$ not containing $h_b$, and hence $\Bbbk G_bx_i\subset K_b$. Therefore, $\sum_{i\in E}\Bbbk G_bx_i\subset K_b$ and hence ${\bf 1}_\theta\neq\sum_{i\in E}x_i$ which completes the proof.
\end{proof}

\noindent For each $\theta\in\widehat{\bf T}$, we denote $\mathbb{L}(\theta)$ be the unique simple quotient determined by Theorem \ref{hd=sim}.

\begin{Cor}\label{unique}
Let $\theta\in\widehat{\bf T}$. Then $\mathbb{L}(\theta)$ is the unique $($up to isomorphism$)$ irreducible $\Bbbk{\bf G}$-module containing a ${\bf B}$-stable line associated to $\theta$. Moreover, $\mathbb{L}(\theta)$ contains a unique ${\bf B}$-stable line, and hence $\mathbb{L}(\theta)\simeq\mathbb{L}(\theta')$ if and only if $\theta\simeq\theta'$.
\end{Cor}
\begin{proof}
Clearly, any such irreducible module is a quotient of $\mathbb{M}(\theta)$ and hence isomorphic to $\mathbb{L}(\theta)$ by Theorem \ref{hd=sim}. The second statement follows from
$$\dim\operatorname{Hom}_{\Bbbk{\bf B}}(\theta,\mathbb{L}(\theta))=\dim\operatorname{Hom}_{\Bbbk{\bf G}}(\mathbb{M}(\theta),\mathbb{L}(\theta))=1.$$
This completes the proof.
\end{proof}
\noindent Thus, we have classified all simple $\Bbbk{\bf G}$-modules with ${\bf B}$-stable line (in principle) up to  isomorphism.

\section{Construction of simple modules with ${\bf B}$-stable line}
The goal of this section is to construct $\mathbb{L}(\theta)$ explicitly, and give the connection between $\mathbb{L}(\theta)$ and rational representations.

\medskip
\noindent We first deal with the case when $\mathbb{L}(\theta)$ is finite-dimensional.

\begin{Lem}\label{irrdg}
A $\Bbbk{\bf G}$-module is irreducible if and only if it is irreducible as $\Bbbk\mathcal{D}{\bf G}$-module.
\end{Lem}
\begin{proof}
Let $V$ be an irreducible $\Bbbk{\bf G}$-module. The ``if" part is obvious. One has the isogeny $\phi:\mathcal{D}{\bf G}\times{\bf T}''\rightarrow{\bf G}$, where ${\bf T}''$ is a torus. Since $\phi({\bf T}'')$ is contained in the center of ${\bf G}$ and each element of $\phi({\bf T}'')$ has finite order, each element of $\phi({\bf T}'')$ acts on $V$ by a scalar, and hence $V$ remains irreducible when restricted to $\mathcal{D}{\bf G}$. This proves the ``only if" part.
\end{proof}

\begin{Cor}\label{L+}
{\it Let $\theta\in\widehat{\bf T}$, and assume that $\pi(\theta)\in X({\bf T}')^+$. Then $\dim\mathbb{L}(\theta)<\infty$.}
\end{Cor}
\begin{proof}
Lemma \ref{irrdg} implies $\mathbb{L}(\theta)$ is an irreducible $\Bbbk\mathcal{D}{\bf G}$-module. It is clear that $\mathbb{L}(\theta)$ contains a ${\bf B}'={\bf T}'\ltimes{\bf U}$-stable line associated to $\pi(\theta)$. It follows that $\mathbb{L}(\theta)\simeq L(\pi(\theta))$ as $\Bbbk\mathcal{D}{\bf G}$-modules by Corollary \ref{unique}, and hence $\dim\mathbb{L}(\theta)=\dim L(\pi(\theta))<\infty$.
\end{proof}

\begin{Prop}\label{ifpart}
Assume that $\mathcal{D}{\bf G}$ is simply connected semisimple. Let $\theta\in\widehat{\bf T}$ and assume that $\theta|_{{\bf T}_i}\in\mathcal{X}_0$ for any $i\in I$. Then $\mathbb{L}(\theta)\simeq \mathbb{L}(\theta_1)^{\omega_1}\otimes\cdots\otimes \mathbb{L}(\theta_l)^{\omega_l}$ for some $\theta_1,\cdots,\theta_l\in \mathbb{X}_1$ and distinct automorphisms $\omega_1,\cdots,\omega_l\in\mathbb{G}$, and hence $\dim\mathbb{L}(\theta)<\infty$.
\end{Prop}
\begin{proof}
By Lemma \ref{kkk}, one has $\theta=\sum_{i=1}^l\theta_i^{\omega_i}$ for some $\theta_1,\cdots,\theta_l\in X_1$ and distinct automorphisms $\omega_1,\cdots,\omega_l\in\mathbb{G}$. Clearly, $\mathbb{L}(\theta_1)^{\omega_1}\otimes\cdots\otimes \mathbb{L}(\theta_l)^{\omega_l}$ is isomorphic to $L(\pi(\theta_1))^{\omega_1}\otimes\cdots\otimes L(\pi(\theta_l))^{\omega_l}$ as $\Bbbk\mathcal{D}{\bf G}$-modules. Each $\omega_i$ corresponds to the sequence $(g_{i,n})_{n\ge N}$ satisfying assumptions in Lemma \ref{red} (ii). Set $\lambda_n=\sum_{i=1}^lp^{g_{i,n}}\pi(\theta_i)$. Since $\mathcal{D}{\bf G}$ is simply connected semisimple, we have $L(\lambda_n)=L(\pi(\theta_1))^{[g_{1,n}]}\otimes\cdots\otimes L(\pi(\theta_l))^{[g_{l,n}]}$ for $n\ge N$ by the Steinberg's tensor product theorem, where the superscript $[m]$ means $m$-th Frobenius twist. Therefore, $L(\pi(\theta_1))\otimes\cdots\otimes L(\pi(\theta_l))$
is the common underlying space of all $L(\lambda_n)$ $(n\in\mathbb{N})$. Now we obtain a system $\cdots\rightarrow L(\lambda_n)\rightarrow L(\lambda_{n+1})\rightarrow\cdots$, in which each $L(\lambda_n)$ is an irreducible $\Bbbk\mathcal{D}G_n$-module (by Theorem \ref{restrictionthm}) and all maps are identity of underlying space. Clearly, $L(\lambda_n)\rightarrow L(\lambda_{n+1})$ is a $\mathcal{D}G_n$-module homomorphism for all $n$, and hence the direct limit of the system is the $\Bbbk\mathcal{D}{\bf G}$-module $L(\pi(\theta_1))^{\omega_1}\otimes\cdots\otimes L(\pi(\theta_l))^{\omega_n}$ which is irreducible by \cite[Lemma 1.5]{Xi}. It follows that $\mathbb{L}(\theta_1)^{\omega_1}\otimes\cdots\otimes \mathbb{L}(\theta_l)^{\omega_l}$ is an irreducible $\Bbbk{\bf G}$-module by Lemma \ref{irrdg},
and contains a ${\bf B}$-stable line associated to $\theta$. It follows that $\mathbb{L}(\theta)\simeq\mathbb{L}(\theta_1)^{\omega_1}\otimes\cdots\otimes \mathbb{L}(\theta_l)^{\omega_l}$ by Corollary \ref{unique}. By Corollary \ref{L+}, $\dim\mathbb{L}(\theta_i)<\infty~(1\le i\le l)$, and hence $\dim\mathbb{L}(\theta)<\infty$.
\end{proof}

\begin{Rem}\label{r}
\normalfont
(i) We still have $\dim\mathbb{L}(\theta)<\infty$ if $\theta|_{{\bf T}_i}\in\mathcal{X}_0$ for any $i\in I$, without the assumption that $\mathcal{D}{\bf G}$ is simply connected. Indeed, let $\widetilde{\bf G}$ be the simply connected cover of $\mathcal{D}{\bf G}$. We have the natural homomorphism $\varphi: \widetilde{\bf G}\rightarrow{\bf G}$ given by the composition $\widetilde{\bf G}\twoheadrightarrow\mathcal{D}{\bf G}\rightarrow{\bf G}$. By Proposition \ref{ifpart}, $\mathbb{L}(\theta)$ is finite-dimensional when pulled back to $\widetilde{\bf G}$, and hence $\dim\mathbb{L}(\theta)<\infty$.

\medskip
\noindent(ii) Assume that $\theta|_{{\bf T}_i}\in\mathcal{X}_0$ for any $i\in I$ as in (i).  Proposition \ref{ifpart} implies that $\mathbb{L}(\theta)$ is isomorphic to $L(\lambda_n)=L(\pi(\theta_1))^{[g_{1,n}]}\otimes\cdots\otimes L(\pi(\theta_l))^{[g_{l,n}]}$ if $n\ge N$ (keep notation in the proof of Proposition \ref{ifpart}) as $\Bbbk\widetilde{G}_n$-modules which is irreducible by Theorem \ref{restrictionthm}. Choose $a\in\mathbb{N}^*$ so that $\varphi(\widetilde{G}_n)\subset G_a$, then $\mathbb{L}(\theta)$ is an irreducible $\Bbbk G_a$-module.
\end{Rem}

\noindent Now we turn to the general case.
\begin{Thm}\label{paraind}
Let $J=\{i\in I\mid\theta|_{{\bf T}_i}\in\mathcal{X}_0\}$, and ${\bf P}_J={\bf L}_J\ltimes{\bf U}_J$ be the Levi decomposition. Let $\mathbb{L}_J(\theta)$ be the unique irreducible $\Bbbk{\bf L}_J$-module containing the ${\bf B}\cap{\bf L}_J$-stable line associated to $\theta\in\widehat{\bf T}$. Then $\operatorname{Ind}_{\Bbbk{\bf P}_J}^{\Bbbk\bf G}\mathbb{L}_J(\theta)$ is irreducible.
\end{Thm}
\begin{proof}
 For any $0\neq x\in N\in\operatorname{Ind}_{\Bbbk{\bf P}_J}^{\Bbbk\bf G}\mathbb{L}_J(\theta)=\bigcup_a\operatorname{Ind}_{\Bbbk P_{J,a}}^{\Bbbk G_a}\mathbb{L}_J(\theta)$, we have $x\in\operatorname{Ind}_{\Bbbk P_{J,a}}^{\Bbbk G_a}\mathbb{L}_J(\theta)$ for some $a\in\mathbb{N}^*$. We can assume that $a$ is large enough so that $\mathbb{L}_J(\theta)$ is irreducible $\Bbbk L_{J,a}$-module thanks to Remark \ref{r} (ii). Let ${\bf 1}_{\theta,J}$ be a nonzero vector in the ${\bf B}\cap{\bf L}_J$-stable line (which is unique thanks to Corollary \ref{unique}) in $\mathbb{L}_J(\theta)$. We claim that
\begin{equation}\label{cl}
(\operatorname{Ind}_{\Bbbk P_{J,a}}^{\Bbbk G_a}\mathbb{L}_J(\theta))^{U_a}
=\bigoplus_{w\in W^J}\Bbbk\underline{U_{w^{-1},a}}\dot{w}{\bf 1}_{\theta,J}.
\end{equation}
Indeed, the Bruhat decomposition implies that the decomposition
$$\operatorname{Ind}_{\Bbbk P_{J,a}}^{\Bbbk G_a}\mathbb{L}_J(\theta)=\bigoplus_{w\in W^J}\Bbbk U_{w^{-1},a}\dot{w}\mathbb{L}_J(\theta)$$
of $\Bbbk U_a$-modules, and hence
$$(\operatorname{Ind}_{\Bbbk P_{J,a}}^{\Bbbk G_a}\mathbb{L}_J(\theta))^{U_a}=\bigoplus_{w\in W^J}(\Bbbk U_{w^{-1},a}\dot{w}\mathbb{L}_J(\theta))^{U_a}\subset\bigoplus_{w\in W^J}(\Bbbk U_{w^{-1},a}\dot{w}\mathbb{L}_J(\theta))^{U_{w^{-1},a}}.$$
Moreover, for each $w\in W^J$, $U_{w^{-1},a}$ acts freely on $\Bbbk U_{w^{-1},a}\dot{w}\mathbb{L}_J(\theta)$, and hence $(\Bbbk U_{w^{-1},a}\dot{w}\mathbb{L}_J(\theta))^{U_{w^{-1},a}}=\Bbbk \underline{U_{w^{-1},a}}\dot{w}\mathbb{L}_J(\theta)$. It follows that
$$(\operatorname{Ind}_{\Bbbk P_{J,a}}^{\Bbbk G_a}\mathbb{L}_J(\theta))^{U_a}\subset\bigoplus_{w\in W^J}\Bbbk\underline{U_{w^{-1},a}}\dot{w}\mathbb{L}_J(\theta).$$
Let $w\in W^J$, $v\in\mathbb{L}_J(\theta)$, and
\begin{equation}\label{compare}
y=\underline{U_{w^{-1},a}}\dot{w}v=\dot{w}v+\underline{U_{w^{-1},a}\backslash\{1\}}\dot{w}v\in(\operatorname{Ind}_{\Bbbk P_{J,a}}^{\Bbbk G_a}\mathbb{L}_J(\theta))^{U_a}\cap\underline{U_{w^{-1},a}}\dot{w}\mathbb{L}_J(\theta).
\end{equation}
 For any $u\in U_a\cap L_{J,a}$, we have $\dot{w}u\dot{w}^{-1}\in U_{w^{-1},a}'$. It follows that $\dot{w}u\dot{w}^{-1}y=\dot{w}uv+(*)$, where $(*)$ is a combination of elements of the form $u'wv'$ with $u'\in U_{w^{-1},a}\backslash\{1\}$ and $v'\in\mathbb{L}_J(\theta)$. Since $\dot{w}u\dot{w}^{-1}y=y$, one has $uv=v$ by comparing with (\ref{compare}). It follows that $v\in\mathbb{L}_J(\theta)^{U_a\cap L_{J,a}}=\Bbbk{\bf 1}_{\theta,J}$. This proves the ``$\subset$" part of (\ref{cl}). The other inclusion is obvious, and hence (\ref{cl}) is proved. Therefore,
\begin{equation}\label{thanks}
0\ne(\Bbbk G_ax)^{U_a}\subset(\operatorname{Ind}_{\Bbbk P_{J,a}}^{\Bbbk G_a}\mathbb{L}_J(\theta))^{U_a}
=\bigoplus_{w\in W^J}\Bbbk\underline{U_{w^{-1},a}}\dot{w}{\bf 1}_{\theta,J},
\end{equation}
where ``$\ne$" follows from Lemma \ref{fixedpt}. Thanks to (\ref{thanks}), one can choose a nonzero element
$$\sum_{w\in W^J}c_w\underline{U_{w^{-1},a}}\dot{w}{\bf 1}_{\theta,J}\in(\Bbbk G_ax)^{U_a},~c_w\in\Bbbk.$$
By Lemma \ref{forsomew}, we have $\underline{U_{w^{-1},b}}\dot{w}{\bf 1}_{\theta,J}\in N$ for some $w\in W^J$ and $b>a$. Clearly, for any $v\in W^J$ and $i\in I$ with $s_iv>v$ and ${\bf U}_{v^{-1}}=({\bf U}_{v^{-1}})^{s_i}$, we have $v^{-1}(\alpha_i)=\alpha_j$ for some $j\in I\backslash J$. It follows that $\theta^v|_{{\bf T}_i}=\theta|_{{\bf T}_j}\in\mathcal{X}_1$ by assumption. Combining this and applying Lemma \ref{c2} (ii) repeatedly yields ${\bf 1}_{\theta,J}\in N$ and hence $N=\operatorname{Ind}_{\Bbbk{\bf P}_J}^{\Bbbk\bf G}\mathbb{L}_J(\theta)$ which completes the proof.
\end{proof}

\begin{Cor}\label{IndPG}
Let $\theta\in\widehat{\bf T}$ and $J=\{i\in I\mid\theta|_{{\bf T}_i}\in\mathcal{X}_0\}$. Then $\mathbb{L}(\theta)\simeq\operatorname{Ind}_{\Bbbk{\bf P}_J}^{\Bbbk{\bf G}}\mathbb{L}_J(\theta)$ $($keeping the notation in Theorem \ref{paraind}$)$. In particular, $\dim\mathbb{L}(\theta)<\infty$ if and only if $\theta|_{{\bf T}_i}\in\mathcal{X}_0$ for any $i\in I$.
\end{Cor}
\begin{proof}
 It is clear that $\mathbb{M}(\theta)=\operatorname{Ind}_{\Bbbk{\bf P}_J}^{\Bbbk{\bf G}}\operatorname{Ind}_{\Bbbk{\bf B}\cap{\bf L}_J}^{\Bbbk{\bf L}_J}\theta$ by \cite[Lemma 2.2]{Xi}, where we regard $\operatorname{Ind}_{\Bbbk{\bf B}\cap{\bf L}_J}^{\Bbbk{\bf L}_J}\theta$ as a $\Bbbk{\bf P}_J$-module by the homomorphism ${\bf P}_J\rightarrow{\bf L}_J$. Since $\operatorname{Ind}_{\Bbbk{\bf B}\cap{\bf L}_J}^{\Bbbk{\bf L}_J}\theta\twoheadrightarrow\mathbb{L}_J(\theta)$, we have $\mathbb{M}(\theta)\twoheadrightarrow\operatorname{Ind}_{\Bbbk{\bf P}_J}^{\Bbbk{\bf G}}\mathbb{L}_J(\theta)$ by the exactness of $\operatorname{Ind}_{\Bbbk{\bf P}_J}^{\Bbbk{\bf G}}(-)$. Combining Theorem \ref{paraind} and Theorem \ref{hd=sim}, we see that $\operatorname{Ind}_{\Bbbk{\bf P}_J}^{\Bbbk{\bf G}}\mathbb{L}_J(\theta)$ is the unique simple quotient of $\mathbb{M}(\theta)$ and hence $\mathbb{L}(\theta)=\operatorname{Ind}_{\Bbbk{\bf P}_J}^{\Bbbk{\bf G}}\mathbb{L}_J(\theta)$ as desired.
\end{proof}

The module $\mathbb{L}(\theta)$ is called $p$-restricted if it is a rational irreducible module with $p$-restricted highest weight when restricted to $\mathcal{D}{\bf G}$. Proposition \ref{ifpart} implies that $\mathbb{L}_J(\theta)$ in Theorem \ref{paraind} is a twist tensor product of $p$-restricted irreducible ${\bf L}_J$-modules. Combining this and Corollary \ref{IndPG}, we see that
\begin{Thm}
Assume that $\mathcal{D}{\bf G}$ is simply connected semisimple. Then any irreducible $\Bbbk {\bf G}$-module with ${\bf B}$-stable line is a parabolic induction of a twist tensor product of $p$-restricted irreducible modules $($for some Levi subgroup$)$.
\end{Thm}
In other words, {\bf the ``size" of rational $p$-restricted irreducible modules for some Levi subgroup determines the ``size" of all abstract irreducible $\Bbbk{\bf G}$-module with ${\bf B}$-stable line}.

\medskip
Finally, as a byproduct, we give a new proof of the classification result on the finite-dimensional abstract irreducible $\Bbbk{\bf G}$-modules which is a special case of \cite[Theorem 10.3]{BT}.

\begin{Cor}
Assume that ${\bf G}$ is simply connected semisimple. If $V$ is a finite-dimensional $\Bbbk{\bf G}$-module, then
$V\simeq L(\theta_1)^{\omega_1}\otimes\cdots\otimes L(\theta_l)^{\omega_l}$
for some $\theta_1,\cdots,\theta_l\in X_1$ and distinct automorphisms $\omega_1,\cdots,\omega_l\in\mathbb{G}$.
\end{Cor}
\begin{proof}
Let $\mathbb{P}(V)$ be the projective space associated to $V$. Clearly, $\mathbb{P}(V)^{B_n}\ne\varnothing$ and $\mathbb{P}(V)^{B_n}\supset\mathbb{P}(V)^{B_{n+1}}$ for any $n$. All $\mathbb{P}(V)^{B_n}$ are closed in $\mathbb{P}(V)$. It follows that $\mathbb{P}(V)^{\bf B}=\bigcap_n\mathbb{P}(V)^{B_n}\ne\varnothing$ since the chain $\cdots\supset\mathbb{P}(V)^{B_n}\supset\mathbb{P}(V)^{B_{n+1}}\supset\cdots$ is stable ($\mathbb{P}(V)$ is noetherian). Thus, $V$ contains a ${\bf B}$-stable line and hence is a quotient of $\mathbb{M}(\theta)$ for some $\theta\in\widehat{\bf T}$.  So Corollary \ref{IndPG} implies that $\theta|_{{\bf T}_i}\in\mathcal{X}_0$ for any $i\in I$. Combining with Proposition \ref{ifpart} completes the proof.
\end{proof}

\bigskip

\bibliographystyle{amsplain}

\end{document}